\documentclass[11pt,reqno]{amsart}

\usepackage{amsmath,mathtools}
\usepackage{amsthm}
\usepackage{amssymb}
\usepackage{verbatim}
\usepackage[dvipsnames]{xcolor}
\usepackage{hyperref}
\usepackage{mathtools}
\usepackage{enumitem}   
\usepackage{graphicx}
\usepackage{fullpage}

\usepackage{todonotes}
\usepackage{tikz}
\renewcommand{\arraystretch}{1.5}
\usepackage{fancyhdr}
\newcommand\vaa{\frac{\eta^{16}(\tau)}{\eta^6(\tau/2) \eta^6(2\tau)}}
\newcommand\vab{\frac{\eta^{6}(\tau/2)}{\eta^2(\tau)}}
\newcommand\vac{\frac{\eta^{6}(2\tau)}{\eta^2(\tau)}}
\newcommand\vba{\frac{\eta^{7}(\tau)}{\eta^3(\tau/2)}}
\newcommand\vbb{\frac{\eta^{3}(\tau/2)\eta^3(2\tau)}{\eta^2(\tau)}}
\newcommand\vbc{\frac{\eta^{7}(\tau)}{\eta^3(2\tau)}}
\newcommand\vca{\frac{\eta^{7}(\tau)}{\eta^3(2\tau)}}
\newcommand\vcb{\frac{\eta^{7}(\tau)}{\eta^3(\tau/2)}}
\newcommand\vcc{\frac{\eta^{3}(\tau/2)\eta^3(2\tau)}{\eta^2(\tau)}}
\newcommand\ev[1]{\mathfrak{e}_{#1}}
\newtheorem{theorem}{Theorem}[section]
\newtheorem{definition}[theorem]{Definition}
\newtheorem{corollary}[theorem]{Corollary}

\newtheorem{lemma}[theorem]{Lemma}
\newtheorem{prp}[theorem]{Proposition}
\newtheorem{lthm}{Theorem}

\newtheorem{lcor}[lthm]{Corollary}

\newcommand{\R}{\mathbb{R}}

\newcommand{\C}{\mathbb{C}}

\newcommand{\Z}{\mathbb{Z}}

\newcommand{\sign}{\operatorname{sign}}
\newcommand{\holoproj}[1]{\pi_{\mathrm{hol}}\left(#1\right)}
\tikzset{circ/.style = {fill, circle, inner sep = 0, minimum size = 3}}

\newcommand{\im}{\textnormal{Im}}

\renewcommand{\arraystretch}{1.5}

\renewcommand{\Im}{\im}

\usepackage{color}

\usepackage{commath}
\usepackage{float}
\usepackage{enumitem}

\renewcommand{\mod}[1]{\ (\mathrm{mod}\ #1)}
\tikzset{circ/.style = {fill, circle, inner sep = 0, minimum size = 3}}
\usetikzlibrary{arrows.meta}
\usetikzlibrary{decorations.markings}
\usetikzlibrary{decorations.pathmorphing}
\usetikzlibrary{positioning}
\usetikzlibrary{intersections}
\usepackage{tkz-euclide}
\pgfarrowsdeclarecombine{twolatex'}{twolatex'}{latex'}{latex'}{latex'}{latex'}
\pgfarrowsdeclarecombine{twolatex'}{twolatex'}{latex'}{latex'}{latex'}{latex'}
\tikzset{->-/.style = {decoration={markings,
                                   mark=at position #1 with {\arrow[scale=2]{latex'}}},
                       postaction={decorate}}}
\tikzset{-<-/.style = {decoration={markings,
                                   mark=at position #1 with {\arrowreversed[scale=2]{latex'}}},
                       postaction={decorate}}}

\definecolor{sageblue}{rgb}{0,0,1}
\definecolor{sagered}{rgb}{1,0,0}
\definecolor{sagegreen}{rgb}{0.0, 0.5019607843137255, 0.0}
\definecolor{sageorange}{rgb}{1.0, 0.6470588235294118, 0.0}

\newcommand{\bea}{\begin{equation}\begin{aligned}}
\newcommand{\eea}{\end{aligned}\end{equation}}
\usepackage[titletoc]{appendix}

\definecolor{brightube}{rgb}{0.82, 0.62, 0.91}

\newtheoremstyle{myremark}
  {}{}              % Space above/below
  {\normalfont}     % Body font
  {}                % Indent
  {\bfseries}       % Head font
  {.}               % Punctuation after head
  {.5em}             % Space after head
  {}                % Head specification

\theoremstyle{myremark}
\newtheorem{remark}{Remark}

\newtheorem*{example}{Example} 
\newcommand{\dd}{\mathrm{d}}
\newcommand{\sgn}{\operatorname{sgn}}

\usepackage{hhline} %Table
\usepackage{multirow} %Table

 \usepackage{tikz-cd}  %Added by FG
\newcommand{\Span}{\operatorname{Span}}

\usepackage{mleftright}
\mleftright

\allowdisplaybreaks

\title{Some convolution identities for mock modular forms arising from the theory of holomorphic projection}
\author{Jonathan G.~Bradley-Thrush}
\address{Grupo de Física Matemática, Instituto Superior Técnico, Universidade de Lisboa, Av. Rovisco Pais, 1049-001 Lisboa, Portugal}
\email{jonathan.bradley-thrush@cantab.net}
\author{Frank Garvan}
\address{Department of Mathematics, University of Florida, Gainesville, FL 32611-8105}
\email{fgarvan@ufl.edu}
\author{Jayashree Kalita}
\address{Department of Mathematics,
	Vanderbilt University, Nashville, TN 37240}
\email{jayashree.kalita@vanderbilt.edu}
\author{Larry Rolen}
\address{Department of Mathematics,
	Vanderbilt University, Nashville, TN 37240}
\email{larry.rolen@vanderbilt.edu}

\begin{document}

\begin{abstract}
Convolution identities and recursive formulas have long played a role in the theory of holomorphic modular forms and their applications. These have served both as striking formulas and as  fundamentally useful tools for applications to combinatorics. Recently, there has been renewed interest in examples arising from non-holomorphic modular forms. These include the famous Hurwitz-Kronecker class number relations dating to 1885, and groundbreaking work of Imamo\u{g}lu, Raum, and Richter from 2014. Imamo\u{g}lu, Raum, and Richter developed a theory of holomorphic projection for products of (vector-valued) harmonic Maass forms and holomorphic modular forms to produce many such formulas. This was related shortly thereafter by Duncan, Griffin, and Ono to replicable-type functions in the sense of Conway and Norton, and such recursions played a key role in their proof of the Umbral Moonshine Conjecture. 

Here, we develop new results on holomorphic projections of such functions which is more convenient for many natural cases. Imamo\u{g}lu, Raum, and Richter's  choice corresponds to the case in which the generalized Pell equation $m^2 - Dn^2 = N$ has a square value for $D$,  which has only finitely many solutions for a given value of $N$. Our main results cover the cases of general $D$, in particular those for which the Pell equation has infinitely many solutions. 

As an application, we resolve a recent conjecture of the second author. More generally, our approach yields 18 convolution identities for mock theta functions, which we boil down to 5 identities that can directly be used to prove the others. In the appendices, we also show how these identities can be proven by more direct $q$-series methods; however, the key utility of the holomorphic projection formulas is that they give a tool to automatically discover and verify such formulas.
\end{abstract}

\maketitle
%Abstract-------------------------------------------------------------------------------------------------------------------------------------------------------------------------------------------------------------------------------------------------------------------------------------------------------

%Introduction--------------------------------------------------------------------------------------------------------------------------------------------------------------------------------------------------------------------------------------------------------------------------------------------------
\section{Introduction}
\label{sec:intro}

Convolution identities for modular forms have a long history. 
A classical example is the identity
\[
\sigma_7(n)=\sigma_3(n)+120\sum_{0<k<n}\sigma_3(k)\sigma_3(n-k),
\]
where $\sigma_k(n):=\sum_{d|n}d^k$ denotes the $k$-th power divisor-sum function.
This follows immediately from the Eisenstein series relation $E_4^2=E_8$. 
In turn, this relation follows from observing that $E_4^2$ and $E_8$ belong to the
same one-dimensional space of holomorphic modular forms and have the same
constant term.

An important example from $q$-series and combinatorics is given by the integer partition function $p(n)$, which counts the number of ways to write the non-negative integer $n$ as a sum of positive integers. 
Its generating function 
\[
P(q):=\sum_{n\geq0}p(n)q^n,
\]
is essentially a (weakly holomorphic) modular form of weight $-1/2$. 
Euler proved, using the Pentagonal Number Theorem, that 
\[
P(q)\cdot \left(\sum_{n\in\mathbb Z}(-1)^nq^{\frac{n(3n-1)}2}\right)=1
\]
Equating coefficients implies the recursive formula 
\[
p(n)=p(n-1)+p(n-2)-p(n-5)-p(n-7)+\ldots,
\]
which gives an efficient way to compute tables of partition numbers. 

This paper proves a conjectured convolution identity for a mock theta function, stated in Theorem~\ref{thm: main}(a), that was previously posed by the second author \cite{FG}. 
We place this naturally alongside $17$ additional new identities given in 
Appendix~\ref{app: identities}, which are all dictated by the tensor product of
a single vector-valued harmonic Maass form with a vector-valued holomorphic modular form.
The $18$ identities reduce to five as given in our main result, Theorem~\ref{thm: main}.

Roughly speaking, a harmonic Maass form is 
a non-holomorphic function on the upper half-plane which transforms as a modular form, and instead of being holomorphic, satisfies a special second-order differential equation; see Section~\ref{sec:prelim} for more details. 
This differential equation, along with translation invariance coming from the modular transformation, implies that a harmonic Maass form canonically splits into two pieces. 
The first of these, the {\it holomorphic part}, is a classical $q$-series and is termed a mock modular form. 

Convolution identities for mock modular forms have played an important role in the theory of harmonic Maass forms. 
As with many arithmetic properties of mock modular forms, such as congruences, these identities are typically proved by first relating non-holomorphic harmonic Maass forms to classical holomorphic modular forms. 

A key motivational example is given by the Hurwitz-Kronecker class number relation. 
To describe it, let $H(n)$ be the $n$-th Hurwitz class number, which counts classes of positive definite binary quadratic forms of discriminant $-n$ (with a simple weight factor which is usually $1$). 
Zagier showed \cite{Zagier75} that their generating function 
\[
\mathcal H(q):=\sum_{n\geq0}H(n)q^n
\]
is what we would now call a mock modular form of weight $3/2$. Hurwitz~\cite{Hurwitz} and Kronecker~\cite{Kronecker} proved the recurrence relation
\begin{equation}\label{HurwitzKronecker}
\sum_{r\in\Z}H(4n-r^2)-2\lambda_1(n)=2\sigma_1(n)
,
\end{equation}
where
$$
\lambda_k(n):=\frac12\sum_{d|n}\min\left(d,\frac nd\right)^k
.
$$
A modern perspective on this identity using Zagier's mock modularity of $\mathcal H$ is as follows. 
Given a non-holomorphic function transforming as a modular form, the theory of holomorphic projection gives a method to relate it to a holomorphic modular form. This was first defined by Sturm \cite{Sturm474} in 1980, and was an essential tool in Gross-Zagier's work on Heegner points and derivatives of $L$-functions \cite{Zagier241}.
A detailed account of the method will be given in Section~\ref{HolProjBackground}, 
but the rough idea in the present setting is to multiply the harmonic Maass completion of $\mathcal H$ by the Jacobi theta function 
\[
\theta(q):=\sum_{n\in\Z}q^{n^2},
\]
which is a modular form of weight $1/2$, and then project the resulting function to the space of weight $2$ quasimodular forms. 
This projection yields $E_2$, giving the $\sigma_1(n)$ term on the right hand side of \eqref{HurwitzKronecker}. 
The first term on the left hand side is precisely $\mathcal H\cdot\theta$, untouched by holomorphic projection because it is already holomorphic. 
The remaining piece, involving $\lambda_1(n)$, is what arises from the non-holomorphic completion term. 

This approach has blossomed into a powerful general technique in the theory of harmonic Maass forms; see, for example, \cite{ARZ,DMZ,MMR,Mertens384,MOR}.
Although the basic idea is straightforward, its implementation presents several technical hurdles.
In the case of \eqref{HurwitzKronecker}, the resulting projection is quasimodular due to the projection weight being 2, just on the cusp of convergence, requiring the Hecke trick to regularize Sturm's method.
For almost all mock modular forms other than $\mathcal H$, there is exponential growth at the cusps which makes the method even more tricky.
In this setting, one must do serious work to regularize divergent integrals which arise. 
Such an example was first treated in a seminal paper by Imamo\u{g}lu, Raum, and Richter \cite{IRR}.
They studied Ramanujan's third order mock theta function 
\[
f(q):=\sum_{n\geq0}\frac{q^{n^2}}{(-q;q)_n^2}=\sum_{n\geq0}c_f(n)q^n,
\]
where 
\[
(a)_n=(a;q)_n:=\prod_{j=0}^{n-1}\left(1-aq^j\right)
\]
is the usual $q$-Pochhammer symbol. 
Imamo\u{g}lu, Raum, and Richter used regularized holomorphic projection to prove 
\[
\sum_{\substack{m\in\Z\\ 3m^2+m\leq2n}}\left(m+\frac16\right)c_f\left(n-\frac{3}2m^2-\frac{1}{2}m\right)=\frac43\sigma_1(n)-\frac{16}{3}\sigma_1\left(\frac{n}{2
}\right)-2\sum_{\substack{a, b\in\mathbb Z\\ 2n=ab}}d\left(N_1, N_2, \frac16, \frac16\right)
,
\]
where $N_1:=(-3a+b-1)/6$, $N_2:=(3a+b-1)/6$, the sum over $a, b\in \mathbb Z$ is restricted to those pairs of integers for which $N_1, N_2\in\mathbb Z$, and 
%% FG: corrected "t1, t2" --> "t_1, t_2"
\[
d(N_1, N_2, t_1, t_2) := \sign^+(N_1)\sign^+(N_2)\left(|N_1 + t_1| - |N_2 + t_2|\right)
\]
with 
\[
\sign^+(n):=\begin{cases} \sign(n) & \text{ if } n\neq0,\\ 1 & \text{ if } n=0.\end{cases}
\]

This has inspired a large body of work. 
For instance, Mertens~\cite{Mertens384} extended it to obtain recurrence relations involving products of mock modular forms and modular forms, as well as generalizations arising from Rankin-Cohen brackets.
Such recurrences were essential in Duncan-Griffin-Ono's proof of the Umbral Moonshine Conjectures \cite{UmbralMoonshine}. 
They serve as a direct analog of the theory of replicable functions of Conway-Norton \cite{ConwayNorton}, which was essential in monstrous Moonshine. 

Although such techniques are well-known, carrying out the computational details, especially for vector-valued harmonic Maass forms of long length, is highly tedious. 
To streamline these computations, we establish Theorem~\ref{thm:hol_proj_eval}, which facilitates the proof of a set of $q$-series identities that arise from the theory of holomorphic projection.
In Section~\ref{sec:transrep}, we form the tensor product of a vector valued harmonic Maass form with a vector valued modular form.
The resulting vector has dimension $18$, and its holomorphic projection yields $18$ identities.
However, we find that many of them are equivalent by means of classical relations between mock theta functions. 
These equivalences are described in Appendix~\ref{sec:otherids}. 
After accounting for them, we obtain the five distinct identities stated in our main result, Theorem~\ref{thm: main}.
The complete set of all 18 identities is provided in Appendix~\ref{app:all id}.
These identities involve the following third order mock theta functions of Ramanujan:
\begin{align} \label{eq: RMTF}
&
\phi(q)\coloneqq \sum_{n\geq 0}\frac{q^{n^2}}{(-q^2;q^2)_n},
&&
\psi(q)\coloneqq \sum_{n>0}\frac{q^{n^2}}{(q;q^2)_n},
&&
\nu(q)\coloneqq \sum_{n\geq 0}\frac{q^{n^2+n}}{(-q;q^2)_{n+1}}\, .
\end{align}

With these definitions in place, our main result is the following. 
In the identities below, $\left({a \over b}\right)$ denotes the Kronecker symbol.
\begin{lthm}
\label{thm: main}
The following convolution-type $q$-series identities hold.
\begin{enumerate}[label=\alph*., ref=\alph*, itemsep=1em]
\renewcommand{\labelenumi}{\textup{\alph{enumi}.}}
\item $\displaystyle
\begin{aligned}[t]
(q;q)_\infty^3\,\psi(q)
&=
\frac{1}{6}\,
\frac{(q^2;q^2)_\infty^7}
{(q^4;q^4)_\infty^3} - \frac{1}{2}\,
\sum_{m=1}^\infty
\sum_{n=1}^{\left\lfloor \frac{3}{2}m \right\rfloor}
\left(\frac{-4}{m}\right)
\left(\frac{n}{6}\right)
\left(m-\frac{2}{3}n\right)
q^{\frac{m^2}{8}-\frac{n^2}{24}-\frac{1}{12}},\label{eqn:psi2}
\end{aligned}
$

\item $\displaystyle
\begin{aligned}[t]
(q^4;q^4)_\infty^3\,\psi(q)
&=
\frac{1}{3}\,
\frac{(q^2;q^2)_\infty^7}
{(q;q)_\infty^3}
-\frac{1}{2}\,
\sum_{m=1}^\infty\sum_{n=1}^{3m}\left(\frac{-4}{m}\right)\left(\frac{n}{6}\right)\left(m-\frac{1}{3}n\right)q^{\frac{m^2}{2}-\frac{n^2}{24}-\frac{11}{24}},\label{eqn:psi3}
\end{aligned}
$

\item $\displaystyle
\begin{aligned}[t]
(q;q)_\infty^3\,\psi(-q)
&=
\frac{1}{4}\,
\frac{(q;q)_\infty^6}
{(q^2;q^2)_\infty^2}
-\frac{1}{2}\,
\sum_{m=1}^\infty
\sum_{n=1}^{\left\lfloor \frac{3}{2}m \right\rfloor}
\left(\frac{-4}{m}\right)
\left(\frac{n}{12}\right)
\left(m-\frac{1}{2}n\right)
q^{\frac{m^2}{8}-\frac{n^2}{24}-\frac{1}{12}},\label{eqn:psi5}
\end{aligned}
$

\item $\displaystyle
\begin{aligned}[t]
(q;q)_\infty^3\,\nu(q)
&=
\frac{4}{3}\,
\frac{(q;q)_\infty^3\,(q^4;q^4)_\infty^3}
{(q^2;q^2)_\infty^2}
-\sum_{m=1}^\infty
\sum_{n=1}^{\left\lfloor \frac{3}{8}m \right\rfloor}
(-1)^n
\left(\frac{-4}{m}\right)
\left(\frac{n}{3}\right)
\left(m-\frac{8}{3}n\right)
q^{\frac{m^2}{8}-\frac{2n^2}{3}-\frac{11}{24}},\label{eqn:nu2}
\end{aligned}
$

\item $\displaystyle
\begin{aligned}[t]
(q^4;q^4)_\infty^3\,\nu(q)
&=
\frac{(q^4;q^4)_\infty^6}
{(q^2;q^2)_\infty^2}
-\sum_{m=1}^\infty
\sum_{n=1}^{\left\lfloor \frac{3m}{4} \right\rfloor}
(-1)^{n}
\left(\frac{-4}{m}\right)
\left(\frac{n}{3}\right)
(m-n)
q^{\frac{m^2}{2}-\frac{2n^2}{3}-\frac{5}{6}}.\label{eqn:nu3}
\end{aligned}
$

\end{enumerate}
\end{lthm}

\begin{remark}
Identity \eqref{eqn:psi2} was conjectured by the second author 
\cite{FG}. 
We also note Jacobi's identity
\begin{equation} \label{eqn:Jacobi_identity}
(q;q)_\infty^3 = \sum_{n=0}^\infty (-1)^n (2n +1) q^{n(n+1)/2}.
\end{equation}
This means that the left side of the identities in Theorem \ref{thm: main} 
can be written in a form analogous to the left side of 
\eqref{HurwitzKronecker} or the left side of Imamo\u{g}lu, Raum, and Richter's identity for $c_f(n)$. 
For example, for identity \eqref{eqn:psi2} we have
$$
(q;q)_\infty^3\,\psi(q)=
\sum_{n=0}^\infty 
\sum_{\substack{m\ge0\\ m(m+1)\leq2n}}
(-1)^n (2n+1) c_\psi\left(n - \frac{1}{2}m(m+1)\right) q^n.
$$
\end{remark}
\begin{remark}
Note that the functions $\phi$ and $\psi$ are related by the identity
\begin{equation} \label{eqn:phi-psi_relation}
\phi(q) + 2\psi(q) = \frac{(q^2;q^2)^7_\infty}{(q;q)^3_\infty (q^4;q^4)_\infty^3}
=\frac{(-q;-q)_\infty^3}{(q^2;q^2)^2_\infty},
\end{equation}
which follows from a pair of identities stated by Ramanujan~\cite[p.~354]{Ramanujan1927} and was later proved by Watson~\cite[p.~70]{Watson1936}. Consequently, the first three identities of Theorem~\ref{thm: main} may be expressed in terms of either $\phi$ or $\psi$. Likewise, since the mock theta function $\omega$ defined by
\begin{equation} \label{eqn:omega_def}
\omega(q) := \sum_{n=1}^\infty {q^{2n(n-1)} \over (q;q^2)_n^2}
\end{equation}
is related to $\nu$ \cite[p.~72]{Watson1936} by
\begin{equation} \label{eqn:nu-omega_relation}
\nu(q) + q\omega(q^2) = (-q^2;q^2)_\infty^3 (q^2;q^2)_\infty,
\end{equation}
we may express the identities~\eqref{eqn:nu2} and~\eqref{eqn:nu3} in an equivalent form in terms of $\omega$.

It is also convenient to note that~\eqref{eqn:nu-omega_relation} immediately implies
\begin{equation} \label{eqn:even_part_of_nu}
\nu(q) + \nu(-q) = 2\frac{(q^4;q^4)^3_\infty}{(q^2;q^2)^2_\infty}.
\end{equation}
Using the relations~\eqref{eqn:phi-psi_relation} and~\eqref{eqn:even_part_of_nu}, we may rewrite the identities in Theorem~\ref{thm: main} in an alternative form, provided in Corollary~\ref{cor:alt_main} below.
\end{remark}

\begin{lcor} \label{cor:alt_main}
The mock theta functions $\phi$, $\psi$, and $\nu$ satisfy the following relations.

\begin{enumerate}[label=\alph*., ref=\alph*, itemsep=1em]
\renewcommand{\labelenumi}{\textup{\alph{enumi}.}}
\item $\displaystyle
\begin{aligned}[t]
\phi(q) - 4\psi(q) &= \frac{1}{(q;q)^3_\infty}\,\sum_{m=1}^\infty\sum_{n=1}^{\left\lfloor \frac{3}{2}m \right\rfloor}\left(\frac{-4}{m}\right)\left(\frac{n}{6}\right)\left(3m-2n\right)q^{\frac{m^2}{8}-\frac{n^2}{24}-\frac{1}{12}}, \label{eqn:Ba} \\
\end{aligned}
$

\item $\displaystyle
\begin{aligned}[t]
\phi(q)-\psi(q)
&=
\frac{1}{2}\,
{1 \over (q^4;q^4)_\infty^3}\sum_{m=1}^\infty\sum_{n=1}^{3m}
\left(\frac{-4}{m}\right)\left(\frac{n}{6}\right)\left(3m-n\right)
q^{\frac{m^2}{2}-\frac{n^2}{24}-\frac{11}{24}}, \label{eqn:Bb} \\
\end{aligned}
$

\item $\displaystyle
\begin{aligned}[t]
\phi(-q) - 2\psi(-q) &= \frac{1}{(q;q)^3_\infty}\,\sum_{m=1}^\infty\sum_{n=1}^{\left\lfloor \frac{3}{2}m \right\rfloor}\left(\frac{-4}{m}\right)\left(\frac{n}{12}\right)\left(2m-n\right)q^{\frac{m^2}{8}-\frac{n^2}{24}-\frac{1}{12}}, \label{eqn:Bc} \\
\end{aligned}
$

\item $\displaystyle
\begin{aligned}[t]
2\nu(-q) - \nu(q) &= \frac{1}{(q;q)^3_\infty}\sum_{m=1}^\infty\sum_{n=1}^{\left\lfloor \frac{3}{8}m \right\rfloor} (-1)^n \left(\frac{-4}{m}\right)\left(\frac{n}{3}\right)\left(3m-8n\right)q^{\frac{m^2}{8}-\frac{2n^2}{3}-\frac{11}{24}}, \label{eqn:Bd} \\
\end{aligned}
$

\item $\displaystyle
\begin{aligned}[t]
\nu(-q)-\nu(q) &= 2\,{1 \over (q^4;q^4)_\infty^3}\sum_{m=1}^\infty
\sum_{n=1}^{\left\lfloor \frac{3m}{4} \right\rfloor}
(-1)^{n}
\left(\frac{-4}{m}\right)
\left(\frac{n}{3}\right)
(m-n)
q^{\frac{m^2}{2}-\frac{2n^2}{3}-\frac{5}{6}}. \label{eqn:Be}
\end{aligned}
$

\end{enumerate}
\end{lcor}

\begin{remark}
Although the identities of Corollary~\ref{cor:alt_main} are equivalent to those of Theorem~\ref{thm: main} using the classical identities~\eqref{eqn:phi-psi_relation} and~\eqref{eqn:even_part_of_nu}, it is nonetheless possible for us to derive the relations~\eqref{eqn:phi-psi_relation} and~\eqref{eqn:even_part_of_nu} directly via holomorphic projection. 
Consequently, the five identities in Corollary~\ref{cor:alt_main} can likewise be established by holomorphic projection; see Remark~\ref{rmk:proj_lin_comb} in Appendix~\ref{app: identities}.
\end{remark}

\begin{remark}
In his Bourbaki lecture on mock theta functions \cite[p.986-15]{Zagier2009}, Zagier briefly mentioned some analogous results for Ramanujan's seventh order mock theta functions and hinted a connection to holomorphic projection. 
He also mentioned the existence of higher order mock theta functions. 
In his talk at the conference, \textit{Mock theta functions and applications in combinatorics, algebraic geometry, and mathematical physics} (Max Planck Institute, May 2009), he detailed how this construction could be achieved using holomorphic projection.
These $p$-order mock theta functions are related to the Dyson rank function modulo $p$ in the same way that Ramanujan's fifth and seventh order mock theta functions are related to the rank modulo $5$ and $7$ in the Mock Theta Conjectures, which were proved by Hickerson~\cite{Hickerson1988a,Hickerson1988b}. 
The second author~\cite{FG} explained this connection. 
In a subsequent paper, we detail how the methods developed here extend to these $p$-order mock theta functions, the Dyson rank modulo $p$, and the corresponding $p$-order Mock Theta Conjectures.
\end{remark}

The remainder of the paper is organized as follows. 
In Section~\ref{sec:prelim}, we recall the necessary background on holomorphic projection, beginning with the ideas of Sturm and concluding with the results of Imamo\u{g}lu, Raum, and Richter. 
Section~\ref{Transformations} reviews the vector-valued transformations needed for Theorem~\ref{thm: main}, based in particular on recent results of Klein and Kupka. 
In Section~\ref{RepnTheory}, we develop the required representation theory and describe the holomorphic modular projections used in the proof of Theorem~\ref{thm: main}.
Section~\ref{sec:formulaholoproj} explains and proves our general holomorphic projection formulas.
The proof of Theorem~\ref{thm: main} is given in Section~\ref{sec:proofmain}. 

We conclude with three appendices. Appendix~\ref{ProcGuide} gives details on how we found the representation spaces in Section~\ref{sec:transrep} and gives relevant numerical data and computer code used.
Appendix~\ref{app: identities} lists all 18 identities obtained from Theorem~\ref{thm: equality} and discusses equivalences among them. 
Finally, Appendix~\ref{sec:eleproofs} presents alternate, elementary $q$-series proofs of several of the identities in Theorem~\ref{thm: main}.

%Acknowledgement *********** To be inserted after refereeing ************

%Preliminaries-------------------------------------------------------------------------------------------------------------------------------------------------------------------------------------------------------------------------------------------------------------------------------------------------
\section{Preliminaries}
\label{sec:prelim}
%Harmonic Maass forms------------------------------------------------------------------------------------------------------------------------------------------------------------------------------------------------------------------------------------------------------------------------------------------
\subsection{Definitions of key automorphic objects}
In this section, we recall the basic definitions and facts concerning modular forms, harmonic Maass forms, and related automorphic objects that are needed to describe the proof of Theorem~\ref{thm: main}.
Further details on the general theory of harmonic Maass forms may be found in \cite{HMFBook}.
Throughout, $\tau$ will denote a variable in the upper half-plane $\mathbb H$, and we set $q:=e^{2\pi i\tau}$. 
Modular objects have symmetries controlled by matrices in 
\[
\mathrm{SL}_2(\Z):=\left\{\left(\begin{smallmatrix}a& b\\ c & d\end{smallmatrix}\right) : \ a,b,c,d\in\Z, \, ad-bc=1\right\}.
\]
This group also acts on the upper half-plane via fractional-linear transformations:
\[
\left(\begin{smallmatrix}a& b\\ c & d\end{smallmatrix}\right)\tau:=\frac{a\tau+b}{c\tau+d}.
\]
We will require forms with half-integral weights $k$. 
To this end, we recall that the metaplectic cover $\mathrm{Mp}_2(\Z)$ of $\mathrm{SL}_2(\Z)$ consists of pairs $(\gamma,\varepsilon(\tau))$ such that $\gamma=\left(\begin{smallmatrix}a&b\\ c&d\end{smallmatrix}\right)\in\mathrm{SL}_2(\Z)$ and $\varepsilon(\tau)$ is a holomorphic choice of a square root of $c\tau+d$. 
This comes equipped with a group law defined by 
\[
(\gamma_1,\varepsilon_1(\tau))\cdot(\gamma_2,\varepsilon_2(\tau)):=(\gamma_1\gamma_2,\varepsilon_1(\gamma_2(\tau))\varepsilon_2(\tau)).
\]
For simplicity, we will sometimes write $\gamma$ for an element of $\mathrm{Mp}_2(\Z)$.

Let $\rho$ be a finite-dimensional unitary representation of $\mathrm{Mp}_2(\Z)$, and denote its representation space by $V(\rho)$.
If we pick a basis $\{\mathfrak e_1,\ldots,\mathfrak e_d\}$ of $V(\rho)$, 
then we can represent any function 
$F\colon\mathbb H\rightarrow V(\rho)$ by a sum
\[
F=:\sum_{j=1}^d F_j\mathfrak e_j. 
\]
\noindent
For any element $(\gamma,\varepsilon(\tau))\in\mathrm{Mp}_2(\Z)$, we define the 
{\it Petersson slash action in weight $k\in\frac12\Z$}
by 
\[
\left(F\big|_{k,\rho}\gamma\right)(\tau):=\rho(\gamma)^{-1}\varepsilon(\tau)^{-2k}
F(\gamma\tau).
\]
\begin{definition}
Let $\rho$ be a finite dimensional unitary representation of $\mathrm{Mp}_2(\Z)$, and let $F\colon\mathbb H\rightarrow V(\rho)$. We say that $F$ {\bf transforms as a modular form of weight $k$ for $\rho$} if 
\[
F\big|_{k,\rho}\gamma=F \quad\quad\text{ for all }  \gamma\in\mathrm{Mp}_2(\Z). 
\]
We say that $F$ is a {\bf holomorphic modular form of weight $k$ for $\rho$} if the following hold:
\begin{enumerate}
\renewcommand{\labelenumi}{\textup{(\arabic{enumi})}}
    \item The function $F$ transforms as a modular form of weight $k$ for $\rho$. 
    \item The functions $F_j$ are holomorphic. 
    \item The functions $F_j$ are bounded as $\tau\rightarrow i\infty$.
\end{enumerate}
We denote the space of holomorphic modular forms of weight $k$ for $\rho$ by $M_k(\rho)$.
\end{definition}
If we strengthen condition (3) by requiring  that the functions $F_j$ tend to $0$ as $\tau\rightarrow i\infty$, then we say that $F$ is a {\bf cusp form}. 
We denote the space of cusp forms of weight $k$ for $\rho$ by $S_k(\rho)$.

To define harmonic Maass forms, we require the weight $k$ hyperbolic Laplacian, defined by 
\[\Delta_k:=-v^2\left(\frac{\partial^2}{\partial u^2}+\frac{\partial^2}{\partial v^2}\right)+ikv\left(\frac{\partial}{\partial u}+i\frac{\partial}{\partial v}\right),
\]
where $\tau=:u+iv$ with $u,v\in\R$.
\begin{definition}
We say that $F$ is a {\bf harmonic Maass form of weight $k$ for $\rho$} if the following hold:
\begin{enumerate}
\renewcommand{\labelenumi}{\textup{(\arabic{enumi})}}
    \item The function $F$ transforms as a modular form of weight $k$ for $\rho$.
    \item We have $\Delta_k(F_j)=0$  for each $j$.
    \item There exist polynomials $P_{F_j}(\tau)\in\C[q^{-1}]$ such that for each $j$, we have \[
    F_j(\tau)-P_{F_j}(\tau)=O\left(e^{-\varepsilon v}\right) \quad \text{ as } v\rightarrow\infty
    \]
    for some $\varepsilon>0$. 
\end{enumerate}
We denote the space of harmonic Maass forms of weight $k$ for $\rho$ by $H_k(\rho)$.
\end{definition}

Finally, quasimodular forms are defined as the ``holomorphic parts'' of almost holomorphic modular forms. 
More precisely, we have the following definition.
\begin{definition}
We say that a holomorphic function $F\colon\mathbb H\rightarrow V(\rho)$ is a {\bf quasimodular form of weight $k$ for $\rho$} if there is a set $\{F^{(n)}\}_{n=0}^d$ of functions $F^{(n)}\colon\mathbb H\rightarrow V(\rho)$ with $F^{(0)}:=F$ such that
\begin{enumerate}
\renewcommand{\labelenumi}{\textup{(\arabic{enumi})}}
    \item The ``completion'' 
        \[
        \sum_{n=0}^d v^{-n}\cdot F^{(n)}
        \]        
    transforms as a modular form of weight $k$ for $\rho$.
    \item  Each function $F^{(n)}$ is bounded as $v\rightarrow\infty$.
    \end{enumerate} 
We denote the space of quasimodular forms of weight $k$ for $\rho$ by $\widetilde M_k(\rho)$.
\end{definition}

Modular forms are frequently studied via their Fourier expansions. Vector-valued modular and quasimodular forms have simple Fourier expansions: each component $F_j$ is a classical $q$-series 
\[
\sum_{n\geq0}a_{F_j}(n)q^n.
\]
This arises from translation invariance and holomorphicity. 
For harmonic Maass forms, the second-order differential equation in part~(2) of the definition can be combined with translation invariance to give Fourier expansions.  
The two linearly independent solutions of this differential equation give rise to a canonical decomposition of a harmonic Maass form into two pieces.
Specifically, each component function $F_j$ can be decomposed as 
\[
F_j=F_j^+ + F_j^-.
\]
We refer to $F_j^+$ as the {\it holomorphic part} of $F_j$ and we call $F_j^-$ its non-holomorphic part. 
The holomorphic part has a Fourier expansion which is a classical $q$-series of the shape\footnote{Here, $n\gg-\infty$ means that $n\geq -N$ for a fixed $N\in\mathbb N$. 
Similarly, $n\ll\infty$ means that $n\leq N$ for some $N$.}
\[
F_j^+=\sum_{n\gg -\infty} a_{F_j^+}(n)q^n.
\]
The non-holomorphic part has a Fourier expansion of the shape
\[
F_j^-=\sum_{n\ll \infty} a_{F_j^-}(n)\Gamma(1-k,-4\pi n v)\,q^n,
\]
where
\[
\Gamma(s,z):=\int_z^{\infty} e^{-t}t^{s}\frac{dt}{t}
\]
is the {\it (upper) incomplete gamma function}.

Finally, we require the {\it shadow operator} of Bruinier and Funke. This is a map $\xi_k\colon H_k(\rho)\rightarrow S_{2-k}(\overline{\rho})$, where the $\overline{\cdot}$ denotes complex conjugation. It is given by 
\[
\xi_k(F):=2i v^k\overline{\frac{\partial F}{\partial\overline{\tau}}}.
\]
The key point for us is that the Fourier expansion of the shadow of $F$ is a reformulation of the Fourier expansion of the non-holomorphic part of $F$. 
Specifically, if $\xi_k(F)=:G$, then for each component we have that
\begin{equation}\label{EichlerInt}
F_j^-=(-4\pi)^{k-1}\sum_{n<0}\overline{a_{G_j}(-n)}(-n)^{k-1}\Gamma(1-k,-4\pi n v)\, q^n=-(2i)^{k-1}\int_{-\overline{\tau}}^{i\infty}\frac{\overline{G_j(-\overline{w})}}{(w+\tau)^k}\, dw,
\end{equation}
where the last expression is called a {\it non-holomorphic Eichler integral of $G_j$.}

\subsection{Holomorphic projection}\label{HolProjBackground} 
We now review the theory of holomorphic projection.
The basic idea is simple. 
Consider the Petersson inner product on a space of cusp forms $S_k(\Gamma)$:
\[
\langle f, g\rangle:=\int_{\Gamma\backslash \mathbb H}f(\tau)\overline{g(\tau)}
v^k\frac{du\,dv}{v^2}.
\]
Here, $\tau=:u+iv$ with $u,v\in\mathbb R$. 
This is well-defined because the modularity of $f$ and $g$ imply that $f(\tau)\overline{g(\tau)}v^k$ is invariant under the action of $\Gamma$, and the measure $du\,dv\, v^{-2}$ is $\operatorname{SL}_2(\mathbb Z)$-invariant. 
Moreover, this turns the $\C$-vector space $S_k(\Gamma)$ into a Hilbert space. 
Now suppose we have a function $F\colon\mathbb H\rightarrow\C$ which transforms as a modular form of weight $k$ on $\Gamma$, but which is not necessarily holomorphic. 
If we further suppose that $F$ grows at most polynomially at the cusps, then the inner product $\langle f,\,F\rangle$ still converges for each $f\in S_k(\Gamma)$. 
This defines a linear functional on $S_k(\Gamma)$ given by $f\mapsto\langle f,\,F\rangle$. 
Thus, by the {\it Riesz representation theorem}, there exists a unique function $g\in S_k(\Gamma)$ such that
\[
\langle f,F\rangle =\langle f,g\rangle \quad \text{ for all } f\in S_k(\Gamma). 
\]
Naturally, this element $g$ is holomorphic, and if $F$ was already holomorphic, then $F=g$. 
Thus, this defines a projection operator onto the space of holomorphic cusp forms. 
In this case, we write
\[
\pi_{\text{hol}}(F):=g,
\]
and call this the {\it holomorphic projection of $F$}. 
This procedure can be made very explicit using the theory of Poincar\'e 
series and the Petersson coefficient formula \cite[Lemma 14.3, p.359]{EK-HI-book}, which states that the inner product of these special cusp forms extract Fourier coefficients of general cusp forms. 

There are two complications with applying this classical framework to our situation.
In the case of mock theta functions, the harmonic Maass form completions automatically have exponential growth at the cusps; see Theorem 1.1 of \cite{GOR}. 
Thus, the argument needs to deal with this growth. This is achieved by multiplying by a cusp form to cancel the poles. 
It is from this step that the convolution sums ultimately arise, both in the present paper and in the works cited above. 
Secondly, our mock theta functions are associated with congruence subgroups of higher level. 
It is therefore much easier to work with them after passing to vector-valued objects, since this allows us to work with the full modular group $\operatorname{SL}_2(\Z)$. 
As we shall see in the proof of Theorem~\ref{thm: equality}, this also dramatically simplifies the finite check needed to establish the vector-valued identity by comparing coefficients, once we know that both sides of the identity live in the same space.

Imamo\u{g}lu, Raum, and Richter~\cite{IRR} developed a particularly general formulation of holomorphic projection that is convenient for our purposes. 
We recall their Definition~2 below.

\begin{definition}
\label{def:IRRdef}
Let $V$ be a finite dimensional complex vector space. Suppose that $F\colon\mathbb{H}\to V$ is continuous with Fourier expansion
\[
F(\tau)=\sum_{m\in\mathbb{Q}} c_F(m,v)q^m,
\]
and assume that $F$ satisfies the following:
\begin{enumerate}
\renewcommand{\labelenumi}{\textup{(\arabic{enumi})}}
\item There is an $\varepsilon>0$ and a
$c_0\in V$ such that as $v \to \infty$, we have
\[
F(\tau)
=
c_0
+
O\!\left(v^{-\varepsilon}\right).
\]
\item For all $m>0$, we have as $v\to0$,
\[
c_F(m,v)=O(v^{2-k}).
\]
\end{enumerate}
Define the holomorphic projection of $F$ in weight $k$ by
\[
\pi_{\mathrm{hol}}^{(k)}(F)(\tau)
:=
c_0
+
\sum_{0<m\in\mathbb{Q}} c_m\,q^m,
\]
where for $m>0$
\[
c_m
:=
\frac{(4\pi m)^{k-1}}{\Gamma(k-1)}
\int_0^\infty
c_F(m,v)e^{-4\pi mv}v^{k-2}\,dv,
\]
and $\Gamma$ is the Gamma-function.
\end{definition}
\begin{remark}
    When it is clear from context, we usually supress the dependence of the holomorphic projection operator on the weight, simply writing $\pi_{\mathrm{hol}}$ for $\pi_{\mathrm{hol}}^{(k)}$.
\end{remark}
Using this definition, Imamo\u{g}lu, Raum, and Richter proved the following result.
\begin{theorem}[\cite{IRR} Theorem 5]\label{MainHolProjThm}
Fix $2\leq k\in\frac12\Z$ and a representation $\rho$ of $\mathrm{Mp}_2(\Z)$. Let $F\colon \mathbb H\rightarrow V(\rho)$ be a continuous function satisfying the following conditions:
\begin{enumerate}
\renewcommand{\labelenumi}{\textup{(\arabic{enumi})}}
\item The function $F$ transforms as a modular form of weight $k$ for the representation $\rho$.
\item We have the growth condition
\[
F(\tau)=c_0+O\left(v^{-\varepsilon}\right),
\]
for some $\varepsilon>0$, $c_0\in V(\rho)$, and as $v\rightarrow\infty$. 
\end{enumerate}
If $k>2$, then $\pi_{\mathrm{hol}}(F)\in M_k(\rho)$, and if $k=2$, then $\pi_{\mathrm{hol}}(F)\in \widetilde{M_2}(\rho)$.
\end{theorem}

In our situation, we consider the case in which $F$ is a tensor product $f\otimes g$ of a vector-valued harmonic Maass form $f$ of length $m$ and a vector-valued holomorphic modular form $g$ of legnth $n$. 
This tensor product $f\otimes g$ is the vector 
\[
(f_1g_1, f_1g_2,\ldots f_1g_n,f_2g_1,f_2g_2,\ldots,f_2g_n,\ldots f_mg_n)
\]
consisting of the products $f_i\cdot g_j$ as we range over all components of both $f$ and $g$ in lexicographical order.
As the conditions above indicate, we need to choose $f$ and $g$ so that the sum of their weights is at least $2$; in our application, we will choose them so that the resulting weight is exactly $2$. 
Moreover, we require the product $f\otimes g$ to be bounded at infinity; that is, $g$ will be chosen so as to cancel out the poles of $f$ at the cusps.

If we make these choices, then we can split $f$ into its holomorphic and non-holomorphic parts and use Theorem~\ref{MainHolProjThm}, together with the fact that $f^+\otimes g$ is already holomorphic, to obtain that
\begin{equation}\label{eqn:split hol_proj}
\pi_{\mathrm{hol}}(F)=\pi_{\mathrm{hol}}\left( (f^+ + f^-)\otimes g
\right)=
f^+\otimes g+\pi_{\mathrm{hol}}\left(f^-\otimes g\right)   
\end{equation}
is a quasimodular form of weight $2$. 
The components of the tensor product $f^+\otimes g$ give rise to the left-hand sides of our identities. 
Thus, the proof of Theorem~\ref{thm: main} boils down to computing two key ingredients. 
The first is to efficiently compute the projection $\pi_{\mathrm{hol}}\left(f^-\otimes g\right)$. 
The second is to compute the space of quasimodular forms for our representation (which ends up being a tensor product of two representations coming from $f$ and $g$). 
We shall see that the latter can be accomplished using explicit constructions of modular forms, representation-theoretic computations, and dimension formulas for spaces of modular forms.

We end this section with a discussion of the former task, part of which requires us to solve a generalized Pell equation. 
Although Imamo\u{g}lu, Raum, and Richter give elegant formulas for $\pi_{\mathrm{hol}}\left(f^-\otimes g\right)$, their results are not as convenient for our situation. 
In particular, our $q$-series are of a more general shape than those allowed by theirs; essentially, this is because they chose the natural first choice of $g$ by multiplying by the ``shadow'' of $f$ (this also nicely makes the growth conditions hold). 
This corresponds to the case in which the generalized Pell equation $m^2 - Dn^2 = N$ has $D$ equal to a perfect square; this has only finitely many solutions for a given value of $N$. 
Our approach instead makes use of Theorem~\ref{thm:hol_proj_eval}, which is well-suited to our situation, as it deals with the case in which the generalized Pell equation has infinitely many solutions on account of having a non-square value of $D$.
%Representation theory--------------------------------------------------------------------------------------------------------------------------------------------------------------------------------------------------------------------------------------------------------------------------------------
\section{Transformations and representation theory}
\label{sec:transrep}
\subsection{Transformations}\label{Transformations}
%Define the function H_(3)-------------------------------------------------------------------------------------------------------------------------
Continuing with the strategy described above, we take as our vector-valued harmonic Maass form, in place of the generic $f$ from the previous section, a form recently described by Klein and Kupka~\cite{KK}. 
To define it, we first consider the following vector of renormalized Ramanujan mock theta functions from~\eqref{eq: RMTF}:
\begin{equation}\label{eqn:F_(3)_def}
F_{(3)}(\tau):=
\begin{pmatrix}
q^{-1/48}\,\phi(q^{1/2})\\
q^{-1/48}\,\phi(-q^{1/2})\\
2\,q^{-1/48}\,\psi(q^{1/2})\\
2\,q^{-1/48}\,\psi(-q^{1/2})\\
\sqrt2\,q^{1/6}\,\nu(q^{1/2})\\
\sqrt2\,q^{1/6}\,\nu(-q^{1/2})
\end{pmatrix}.
\end{equation}
Its non-holomorphic completion (recall the non-holomorphic Eichler integrals of \eqref{EichlerInt}) is described by the function 
\begin{equation} \label{eqn:G_(3)_def}
G_{(3)}(\tau):=-\frac{i}{2\sqrt6}\int_{-\overline{\tau}}^{i\infty}\frac{g_{(3)}(z)}{\sqrt{-i(z+\tau)}}dz,
\end{equation}
with
$g_{(3)}:=(g_{(3),0},\ldots,g_{(3),5})^T$ defined componentwise by
\begin{equation} \label{eqn:g_(3)_components}
\begin{aligned}
&
g_{(3),0}(\tau):=-\left(\theta_{12,1}(\tau)+\theta_{12,5}(\tau)+\theta_{12,7}(\tau)+\theta_{12,11}(\tau)\right),
\\&
g_{(3),1}(\tau):=-\left(\theta_{12,1}(\tau)-\theta_{12,5}(\tau)+\theta_{12,7}(\tau)-\theta_{12,11}(\tau)\right),
\\&
g_{(3),2}(\tau):=\theta_{12,1}(\tau)+\theta_{12,5}(\tau)+\theta_{12,7}(\tau)+\theta_{12,11}(\tau),
\\&
g_{(3),3}(\tau):=\theta_{12,1}(\tau)-\theta_{12,5}(\tau)+\theta_{12,7}(\tau)-\theta_{12,11}(\tau),
\\&
g_{(3),4}(\tau):=-\sqrt2 \,\left(\theta_{12,4}(\tau)+\theta_{12,8}(\tau)\right),
\\&
g_{(3),5}(\tau):=\sqrt2 \,\left(\theta_{12,4}(\tau)+\theta_{12,8}(\tau)\right),
\end{aligned}
\end{equation}
where we have used the unary theta functions
\begin{equation} \label{eqn:unary_theta_def}
\theta_{N,a}(\tau):=\sum_{n \equiv a \mod {2N}}n\, \exp\left(\frac{2\pi i n^2\tau}{4N}\right).
\end{equation}
The result of Klein and Kupka that we require is the following.
\begin{theorem}[\cite{KK}, Theorem 3.21]\label{thm:KK}
The sum $H_{(3)}:=F_{(3)}+G_{(3)}$ is a vector-valued harmonic Maass form of weight $1/2$ on $\operatorname{SL}_2(\Z)$ with representation $\rho_{(3)}$ defined by its action on generators:
\begin{align*}
&
\rho_{(3)}(T):=
\begin{pmatrix}
0 & \zeta^{-1}_{48} & 0 & 0 & 0 & 0 \\
\zeta^{-1}_{48} & 0 & 0 & 0 & 0 & 0 \\
0 & 0 & 0 & \zeta^{-1}_{48} & 0 & 0 \\
0 & 0 & \zeta^{-1}_{48} & 0 & 0 & 0 \\
0 & 0 & 0 & 0 & 0 & \zeta_6 \\
0 & 0 & 0 & 0 & \zeta_6 & 0 \\
\end{pmatrix},
&\quad
\rho_{(3)}(S):= \sqrt{-i}
\begin{pmatrix}
0 & 0 & 1 & 0 & 0 & 0 \\
0 & 0 & 0 & 0 & 0 & 1 \\
1 & 0 & 0 & 0 & 0 & 0 \\
0 & 0 & 0 & 0 & 1 & 0 \\
0 & 0 & 0 & 1 & 0 & 0 \\
0 & 1 & 0 & 0 & 0 & 0 \\
\end{pmatrix}.
\end{align*}
\end{theorem}

In place of the generic modular form $g$ from the previous section, we use a modular form defined by the following vector of eta-quotients, where the Dedekind eta function is given by
\[
\eta(\tau):=q^{\frac1{24}}(q;q)_{\infty}.
\]
%Define the function H-----------------------------------------------------------------------------------------------------------------------------
\begin{prp}
\label{prp:N}
The function 
\\
\begin{equation}\label{eq: H}
N(\tau):=
\begin{pmatrix}
\frac{1}{2}\,\frac{\eta^9(\tau)}{\eta^3(\tau/2)\,\eta^3(2\tau)}\\[0.2em]
\frac{1}{2}\,\eta^3(\tau/2)\\[0.2em]
\sqrt2\,\eta^3(2\tau)
\end{pmatrix},
\end{equation}
\\
is a vector-valued cusp form of weight $3/2$ on $\operatorname{SL}_2(\Z)$ with representation  $\rho$ defined by its action on generators:
\\
\begin{align*}
&
\rho(T):=
\begin{pmatrix}
0 & \zeta_{16} & 0 \\
\zeta_{16} & 0 & 0 \\
0 & 0 & \zeta_{4}  \\
\end{pmatrix},
&\quad
\rho_(S):= \zeta^{-3}_8
\begin{pmatrix}
1 & 0 & 0 \\
0 & 0 & 1 \\
0 & 1 & 0  \\
\end{pmatrix}.
\end{align*}
\end{prp}
\begin{proof}
The proof follows from the fact that $\eta$ is a cusp form of weight $1/2$ satisfying the following functional equations
\begin{align}\label{EtaTransformations}
\eta(\tau + 1)=\zeta_{24} \,\eta(\tau),
&&
\eta\left(-\frac{1}{\tau}\right)=\sqrt{\frac{\tau}{i}}\,\eta(\tau),
&&
\eta\left(\tau + \frac{1}{2}\right)=\zeta_{48}\,\frac{\eta^3(2\tau)}{\eta(\tau)\,\eta(4\tau)}.
\end{align}
\end{proof}

The identities in Appendix~\ref{app: identities} and hence in Theorem~\ref{thm: main} arise from multiplying components of $H_{(3)}$ by components of $N$ and taking their holomorphic projections. 
In order to analyze these holomorphic projections, we need to study the tensor product representation $\rho_3\otimes \rho$. 
We therefore turn to the necessary representation theory.

\subsection{Some representation theory}
\label{RepnTheory}
%The tensor product H_(3)xH ------------------------------------------------------------------------------------------------------------------------

We now examine the tensor product $H_{(3)}\!\otimes N$. 
A straightforward check shows that it satisfies the hypotheses of Theorem~\ref{MainHolProjThm} with $k=2$ and the representation $\rho_{(3)}\!\otimes \rho$. 
Thus, we obtain the following result.
\begin{prp}\label{HolProjSpace}
Assuming the notation above, we have that
\begin{align*}
\pi_{\mathrm{hol}}(H_{(3)}\!\otimes N) \in \widetilde{M}_2(\rho_{(3)}\!\otimes \rho).
\end{align*}
\end{prp}

In order to study the space of quasimodular forms $\widetilde{M}_2(\rho_{(3)}\!\otimes \rho)$, we first decompose the tensor product representation $\rho_{(3)}\otimes\rho$ into irreducible components.

\begin{lemma}\label{lem: Rep}
The representation $\rho_{(3)}\!\otimes\rho$ has a direct sum decomposition $\sigma_1\oplus\ldots\oplus\sigma_6$, where the representation spaces of the irreducible subrepresentations $\sigma_1,\!\ldots,\!\sigma_6$ are spanned, respectively, by the columns of the matrices
\begin{align*}
&
\small
\left( \begin{smallmatrix}
1 & 0 & 0 
\\
 0 & 0 & 0 
\\
 0 & 0 & 0 
\\
 0 & 0 & 0 
\\
 0 & 1 & 0 
\\
 0 & 0 & 0 
\\
 1 & 0 & 0 
\\
 0 & 0 & 0 
\\
 0 & 0 & 0 
\\
 0 & 0 & 0 
\\
 0 & 1 & 0 
\\
 0 & 0 & 0 
\\
 0 & 0 & 0 
\\
 0 & 0 & 0 
\\
 0 & 0 & 1 
\\
 0 & 0 & 0 
\\
 0 & 0 & 0 
\\
 0 & 0 & 1 
\end{smallmatrix} \right),
&&
\small
\left( \begin{smallmatrix}
1 & 0 & 0 
\\
 0 & 0 & 0 
\\
 0 & 0 & 0 
\\
 0 & 0 & 0 
\\
 0 & 1 & 0 
\\
 0 & 0 & 0 
\\
 \text{-}1 & 0 & 0 
\\
 0 & 0 & 0 
\\
 0 & 0 & 0 
\\
 0 & 0 & 0 
\\
 0 & \text{-}1 & 0 
\\
 0 & 0 & 0 
\\
 0 & 0 & 0 
\\
 0 & 0 & 0 
\\
 0 & 0 & \text{-}1 
\\
 0 & 0 & 0 
\\
 0 & 0 & 0 
\\
 0 & 0 & 1
\end{smallmatrix} \right),
&&
\small
\left( \begin{smallmatrix}
0 & 0 & 0 
\\
 0 & 0 & 0 
\\
 1 & 0 & 0 
\\
 0 & 0 & 0 
\\
 0 & 0 & 0 
\\
 0 & 1 & 0 
\\
 0 & 0 & 0 
\\
 0 & 0 & 1 
\\
 0 & 0 & 0 
\\
 0 & 0 & 1 
\\
 0 & 0 & 0 
\\
 0 & 0 & 0 
\\
 1 & 0 & 0 
\\
 0 & 0 & 0 
\\
 0 & 0 & 0 
\\
 0 & 0 & 0 
\\
 0 & 1 & 0 
\\
 0 & 0 & 0 
\end{smallmatrix} \right),
&&
\small
\left( \begin{smallmatrix}
0 & 0 & 0 
\\
 0 & 0 & 0 
\\
 1 & 0 & 0 
\\
 0 & 0 & 0 
\\
 0 & 0 & 0 
\\
 0 & 1 & 0 
\\
 0 & 0 & 0 
\\
 0 & 0 & 1 
\\
 0 & 0 & 0 
\\
 0 & 0 & \text{-}1 
\\
 0 & 0 & 0 
\\
 0 & 0 & 0 
\\
 \text{-}1 & 0 & 0 
\\
 0 & 0 & 0 
\\
 0 & 0 & 0 
\\
 0 & 0 & 0 
\\
 0 & \text{-}1 & 0 
\\
 0 & 0 & 0 
\end{smallmatrix} \right),
&&
\small
\left( \begin{smallmatrix}
0 & 0 & 0 
\\
 1 & 0 & 0 
\\
 0 & 0 & 0 
\\
 1 & 0 & 0 
\\
 0 & 0 & 0 
\\
 0 & 0 & 0 
\\
 0 & 0 & 0 
\\
 0 & 0 & 0 
\\
 0 & 1 & 0 
\\
 0 & 0 & 0 
\\
 0 & 0 & 0 
\\
 0 & 0 & 1 
\\
 0 & 0 & 0 
\\
 0 & 0 & 1 
\\
 0 & 0 & 0 
\\
 0 & 1 & 0 
\\
 0 & 0 & 0 
\\
 0 & 0 & 0 
\end{smallmatrix} \right),
&&
\small
\left( \begin{smallmatrix}
0 & 0 & 0 
\\
 1 & 0 & 0 
\\
 0 & 0 & 0 
\\
 \text{-}1 & 0 & 0 
\\
 0 & 0 & 0 
\\
 0 & 0 & 0 
\\
 0 & 0 & 0 
\\
 0 & 0 & 0 
\\
 0 & 1 & 0 
\\
 0 & 0 & 0 
\\
 0 & 0 & 0 
\\
 0 & 0 & 1 
\\
 0 & 0 & 0 
\\
 0 & 0 & \text{-}1 
\\
 0 & 0 & 0 
\\
 0 & \text{-}1 & 0 
\\
 0 & 0 & 0 
\\
 0 & 0 & 0 
\end{smallmatrix} \right).
\end{align*}
\end{lemma}\,
\begin{proof}
It suffices to find irreducible subspaces invariant under 
$(\rho_{(3)}\!\otimes\rho)(T)$ and $(\rho_{(3)}\!\otimes\rho)(S)$ such 
that their direct sum is the representation space of 
$\rho_{(3)}\!\otimes\rho$. 
We describe a combinatorial approach to this decomposition in Appendix \ref{ap: rep}.
\end{proof}
Next, we obtain an explicit description of the underlying modular form spaces.
%\TD{FG: I have updated the statement and proof of the following Lemma}
\begin{lemma}\label{lem: dim}
Let $\sigma_1,...,\sigma_6$ be as in Lemma \ref{lem: Rep}. 
Then the following are true.
\begin{enumerate}
\renewcommand{\labelenumi}{\textup{(\arabic{enumi})}}
\item For each odd j, we have
$
M_2(\sigma_j)= \Span(V_j),
$
where 
\[
V_1(\tau) \coloneqq 
\begin{pmatrix}
\frac{\eta \left(\tau \right)^{16}}{\eta\left(\tau/2\right)^{6}\, \eta\left(2 \tau \right)^{6}}\\
\frac{\eta\left(\tau/2\right)^{6}}{\eta\left(\tau \right)^{2}}\\
8\,\frac{\eta \left(2 \tau \right)^{6}}{\eta \left(\tau \right)^{2}}
\end{pmatrix},
\quad
V_3(\tau) \coloneqq
\begin{pmatrix}
2\sqrt{2}\,\frac{\eta \left(\tau \right)^{7} }{\eta\left(\tau/2\right)^{3}}\\
2\sqrt{2}\,\frac{\eta\left(\tau/2\right)^{3}\, \eta\left(2 \tau \right)^{3}}{\eta\left(\tau \right)^{2}}\\
\frac{\eta \left(\tau \right)^{7}}{\eta \left(2 \tau \right)^{3}}
\end{pmatrix},
\quad
V_5(\tau) \coloneqq
\begin{pmatrix}
\frac{\eta \left(\tau \right)^{7}}{\eta \left(2 \tau \right)^{3}}\\
2\sqrt{2}\,\frac{\eta \left(\tau \right)^{7}}{\eta\left(\tau/2\right)^{3}}\\
2\sqrt{2}\,\frac{\eta\left(\tau/2\right)^{3} \, \eta \left(2 \tau \right)^{3}}{\eta \left(\tau \right)^{2}}
\end{pmatrix}.
\]
\item For each even j, we have $\mathrm{dim}_\C\left(M_2(\sigma_j)\right)=0$ and $M_2(\sigma_j)=\{0\}$.
\end{enumerate}
\end{lemma}
\begin{proof}
The proof relies on the dimension formula for $k=2$ in 
\cite[p.228]{BR99}. 
If $\sigma$ is a $d$-dimensional unitary representation 
of $\mathrm{Mp}_{2}(\mathbb{Z})$ with $\sigma(S)^2$ equal to the identity, 
then the dimension of the space of holomorphic modular forms of type 
$\sigma$ and weight 2 is
\[
d+\frac{d}{6}-\alpha(-\sigma(S))-\alpha(\zeta_3^{-1}\sigma(ST)^{-1})-\alpha(\sigma(T)),
\]
where for a matrix X with eigenvalues $e^{2\pi i\beta_j}$ \!(0$\leq\beta_j<$1), we define $\alpha(X):=\sum_{j=1}^d\beta_j$.   
\\
%%FG "i s" --> "is"
Note that $\rho_{(3)}\!\otimes\rho$ is a unitary representation, and so is
each subrepresentation $\sigma_i$. 
Using the dimension formula, we find 
(see Table \ref{table} in Appendix \ref{ap: rep})
\begin{align*}
&
\text{dim}_\C\left(M_2(\sigma_j)\right)=1, \text{for $j$ odd},
\\&
\text{dim}_\C\left(M_2(\sigma_j)\right)=0, \text{for $j$ even}.
\end{align*}
Using (\ref{EtaTransformations}) it is routine to show that each
$V_j \in M_2(\sigma_j)$ and the result follows.
\end{proof}
To pass from spaces of modular forms to the space of quasimodular forms that appear in Proposition~\ref{HolProjSpace}, we apply the following handy result from \cite[Proposition~2.2]{IRR}, which is built on previous foundational work of Kaneko and Zagier~\cite{KZ95}.
\begin{prp}\label{NoQuasiForms}
Assume that $\rho$ is an irreducible representation. 
Then we have that
    \[
    \widetilde{M_2}(\rho)=\begin{cases}
        \C\cdot E_2 & \text{ if } \rho \text{ is the trivial representation},
        \\
        M_2(\rho) & \text{ if } \rho \text{ is a  non-trivial representation}.
    \end{cases}
    \]
\end{prp}
We now introduce a candidate for our holomorphic projection, motivated by a numerical comparison of the first several Fourier coefficients (for a tutorial on Maple package for $q$-series computations, see~\cite{Garvan1999}).
We define it by
\begin{equation}\label{eq: E}
% Define the function E ---------------------------------------------------------------------------------------------------------------------------------
\small
\begin{split}
E(\tau):=&\frac{1}{4}(\mathfrak{e}_1+\mathfrak{e}_7)\frac{\eta^{16}(\tau)}{\eta^6(\tau/2)\,\eta^6(2\tau)}+\frac{1}{6}(2\mathfrak{e}_2+2\mathfrak{e}_4+\mathfrak{e}_8+\mathfrak{e}_{10})\frac{\eta^7(\tau)}{\eta^3(2\tau)}
+\frac{\sqrt2}{3}(\mathfrak{e}_3+2\mathfrak{e}_9+\mathfrak{e}_{13}+2\mathfrak{e}_{16})\frac{\eta^7(\tau)}{\eta^3(\tau/2)}
\\&
+\frac{1}{4}(\mathfrak{e}_5+\mathfrak{e}_{11})\frac{\eta^6(\tau/2)}{\eta^2(\tau)}+\frac{\sqrt2}{3}(\mathfrak{e}_6+2\mathfrak{e}_{12}+2\mathfrak{e}_{14}+\mathfrak{e}_{17})\frac{\eta^3(\tau/2)\,\eta^3(2\tau)}{\eta^2(\tau)}+2(\mathfrak{e}_{15}+\mathfrak{e}_{18})\frac{\eta^6(2\tau)}{\eta^2(\tau)},
\end{split}
\end{equation}
where $\mathfrak{e}_1,\ldots,\mathfrak{e}_{18}$ denotes the standard basis for $\mathbb{R}^{18}$.
Using the  transformation formulas from \eqref{EtaTransformations}, it follows that
\begin{align*}
E(\tau+1)=(\rho_{(3)}\!\otimes\rho)(T) E(\tau), && E(-1/\tau)=\tau^2 (\rho_{(3)}\!\otimes\rho)(S) E(\tau).
\end{align*}
By computing the leading coefficients in the $q$-expansions of the components and using the fact that $\eta$ is a cusp form of weight $1/2$, we obtain that $E$ is a cusp form of weight $2$ on $\mathrm{SL}_2(\Z)$ with representation $\rho_{(3)}\!\otimes\rho$.

We are now ready to prove the key identity underlying Theorem~\ref{thm: main}.
\begin{theorem}
\label{thm: equality}
We have that
\begin{equation} \label{eqn:tensor_projection}
\pi_{\mathrm{hol}}(H_{(3)}\!\otimes N)=E.
\end{equation}   
\end{theorem}
\begin{proof}
By Lemmas~\ref{lem: Rep} and \ref{lem: dim} and Proposition~\ref{NoQuasiForms}, we find that $\widetilde{M}_2(\rho_{(3)}\!\otimes \rho)$ and $M_2(\rho_{(3)}\!\otimes \rho)$ both have dimension $3$. 
As the latter is a subspace of the former, we in fact have $\widetilde{M}_2(\rho_{(3)}\!\otimes \rho)=M_2(\rho_{(3)}\!\otimes \rho)$.
By Proposition~\ref{HolProjSpace}, $\pi_{\mathrm{hol}}(H_{(3)}\!\otimes N)\in M_2(\rho_{(3)}\!\otimes \rho)$.
On the other hand, one can directly check that the functions 
\begin{align*}
\mathcal{E}_1(\tau) &:= 
(\ev{1}+\ev{7}) \vaa  + (\ev{5} + \ev{11}) \vab + 8(\ev{15}+\ev{18})\vac,\\
\mathcal{E}_2(\tau) &:= 
(\ev{2}+\ev{4}) \vca  + 2\sqrt{2}(\ev{9} + \ev{16}) \vcb + 2\sqrt{2}(\ev{12}+\ev{14})\vcc,
\\
\mathcal{E}_3(\tau) &:= 
2\sqrt{2} (\ev{3}+\ev{13}) \vba  + 2\sqrt{2}(\ev{6} + \ev{17}) \vbb + (\ev{8}+\ev{10})\vbc
\end{align*}
lie in this space. 
They are also clearly linearly independent, and thus form a basis.
By comparing the first nonzero coefficients of the first three components $\ev{1}$, $\ev{2}$, and $\ev{3}$ (see Appendix~\ref{A2}), we obtain that
\[
\pi_{\mathrm{hol}}(H_{(3)}\!\otimes N)=
\frac{1}{4} \mathcal{E}_1(\tau) + \frac{1}{3} \mathcal{E}_2(\tau)
+ \frac{1}{6} \mathcal{E}_3(\tau).
\]
One directly sees that the right hand side is $E(\tau)$, establishing the claim.
\end{proof}
Consequently, Theorem~\ref{thm: equality} yields $18$ holomorphic projection identities, obtained by equating the corresponding components. 
In particular, equality of the eighth components proves the holomorphic projection version of the conjecture proposed by the 
second author \cite{FG}. 
These identities are listed below as a corollary.
\begin{corollary}
\label{cor:holoprojids}
Let us write $H_{k}$ for the $k$th component of the vector-valued function $H_{(3)}$ in Theorem~\ref{thm:KK}, and $N_j$ for the $j$th component of $N$ in Proposition~\ref{prp:N}; that is, $H_{(3)} =: (H_{0},\ldots,H_{5})^T$ and $N\eqqcolon (N_0, N_1, N_2)^T$.
Then the following holomorphic projection identities hold:\\[1pt]
\begin{enumerate}[label=\arabic*., itemsep=1em]
\renewcommand{\labelenumi}{\textup{\arabic{enumi}.}}
\item $\displaystyle
\begin{aligned}[t]
&
\holoproj{H_{0}\,N_0}(\tau)
=
\holoproj{H_{2}\,N_0}(\tau)
=
\frac{1}{4}\,
\frac{\eta^{16}(\tau)}
{\eta^6(\tau/2)\,\eta^6(2\tau)},
\end{aligned}
$

\item $\displaystyle
\begin{aligned}[t]
&
\pi_{\mathrm{hol}}\left(H_{0}\,N_1\right)(\tau)
=
\pi_{\mathrm{hol}}\left(H_{1}\,N_0\right)(\tau)
=
2\,\pi_{\mathrm{hol}}\left(H_{2}\,N_1\right)(\tau)\\
&=
2\,\pi_{\mathrm{hol}}\left(H_{3}\,N_0\right)(\tau)
=
\frac{1}{3}\,
\frac{\eta^7(\tau)}
{\eta^3(2\tau)},
\end{aligned}
$

\item $\displaystyle
\begin{aligned}[t]
&
2\,\pi_{\mathrm{hol}}\left(H_{0}\,N_2\right)(\tau)
=
\pi_{\mathrm{hol}}\left(H_{2}\,N_2\right)(\tau)
=
2\,\pi_{\mathrm{hol}}\left(H_{4}\,N_0\right)(\tau)\\
&=
\pi_{\mathrm{hol}}\left(H_{5}\,N_0\right)(\tau)
=
\frac{2\sqrt{2}}{3}\,
\frac{\eta^7(\tau)}
{\eta^3(\tau/2)},
\end{aligned}
$

\item $\displaystyle
\begin{aligned}[t]
&
\pi_{\mathrm{hol}}\left(H_{1}\,N_1\right)(\tau)
=
\pi_{\mathrm{hol}}\left(H_{3}\,N_1\right)(\tau)
=
\frac{1}{4}\,
\frac{\eta^6(\tau/2)}
{\eta^2(\tau)},
\end{aligned}
$

\item $\displaystyle
\begin{aligned}[t]
&
2\,\pi_{\mathrm{hol}}\left(H_{1}\,N_2\right)(\tau)
=
\pi_{\mathrm{hol}}\left(H_{3}\,N_2\right)(\tau)
=
\pi_{\mathrm{hol}}\left(H_{4}\,N_1\right)(\tau)\\
&=
2\,\pi_{\mathrm{hol}}\left(H_{5}\,N_1\right)(\tau)
=
\frac{2\sqrt{2}}{3}\,
\frac{\eta^3(\tau/2)\,\eta^3(2\tau)}
{\eta^2(\tau)},
\end{aligned}
$

\item $\displaystyle
\begin{aligned}[t]
&
\pi_{\mathrm{hol}}\left(H_{4}\,N_2\right)(\tau)
=
\pi_{\mathrm{hol}}\left(H_{5}\,N_2\right)(\tau)
=
2\,
\frac{\eta^6(2\tau)}
{\eta^2(\tau)}.
\end{aligned}
$

\end{enumerate}
\end{corollary}

%%%%%%%%%%%%%%%%%%%%%%%%%%%%%%%%%%%%%%%%%%%%%%%%%%%%%%%%%%%%%%%%%%%%%%%%%

\section{A formula for weight-2 holomorphic projections}
\label{sec:formulaholoproj}

In this section, we prove a general holomorphic projection formula, Theorem~\ref{thm:hol_proj_eval}, which will be used in the proof of Theorem~\ref{thm: main}. 
We begin by recalling some facts about generalized Pell equations that are needed for its proof.
Let $D$ be a positive integer which is not a square number. The Pell equation
\begin{equation} \label{eqn:Pell_equation}
m^2 - D n^2 = 1
\end{equation}
admits infinitely many solutions $(m,n) \in \mathbb{Z}^2$. 
Among all solutions of~\eqref{eqn:Pell_equation} with $m,n>0$, the solution $(m,n) = (\alpha,\beta)$ which minimizes the value of $m+n\sqrt{D}$ is called the fundamental solution. 
The complete set of solutions is then given by the union of $\{(u_k,v_k)\}_{k=-\infty}^\infty$ and $\{(-u_k,-v_k)\}_{k=-\infty}^\infty$, where for each $k$, the integers $u_k$ and $v_k$ are determined by $u_k+v_k\sqrt{D} = (\alpha+\beta\sqrt{D}\,)^k$. 
This well-known classical result may be found, for instance, in~\cite[pp.~493--494]{Dirichlet1855} and~\cite[pp.~209--210]{Hermite1851}.

For a given positive integer $j$, two solutions $(m,n), (m',n') \in \mathbb{Z}^2$ of the generalized Pell equation
\begin{equation} \label{eqn:generalized_Pell_equation}
m^2 - Dn^2 = j
\end{equation}
are said to be equivalent if there is some $k \in \mathbb{Z}$ such that
\[ m+n\sqrt{D} = \pm(m'+n'\sqrt{D})(\alpha+\beta\sqrt{D})^k, \]
with $(\alpha,\beta)$ being the fundamental solution of~\eqref{eqn:Pell_equation} defined above. 
This defines an equivalence relation on the set of all integer solutions of~\eqref{eqn:generalized_Pell_equation}; the number of equivalence classes is finite for each $j$~\cite[p.~9]{Frattini1904}. 
The following lemma, from~\cite[p.~396]{Andrews1988}, provides us with a systematic means of choosing a particular representative for each equivalence class.

\begin{lemma}[Andrews, Dyson, and Hickerson] \label{lem:Pell}
Let $j$ be any positive integer and, as above, let $D$ be any positive integer which is not a square number. Let $S_j$ denote the set consisting of all $(m,n) \in \mathbb{Z}^2$ which satisfy not only~\eqref{eqn:generalized_Pell_equation} but also the inequalities
\begin{equation} \label{eqn:Pell_inequalities}
1 \leq m \leq \sqrt{{j \over 2}(\alpha+1)} \quad \text{and} \quad -{\beta \over \alpha+1}m < n \leq {\beta \over \alpha+1}m.
\end{equation}
Then every equivalence class of the set of all integer solutions of~\eqref{eqn:generalized_Pell_equation} contains exactly one element of $S_j$.
\end{lemma}

\begin{remark} \label{rmk:Pell}
{\normalfont
(i) The condition that $m \leq \sqrt{j(\alpha+1)/2}$ is not stated explicitly in~\cite{Andrews1988} but it follows from the other inequalities in~\eqref{eqn:Pell_inequalities} together with the requirement that $(m,n)$ satisfy~\eqref{eqn:generalized_Pell_equation}. 
Indeed,
\[ j = m^2 - Dn^2 \geq m^2 - {D\beta^2 m^2 \over (\alpha+1)^2} = {2m^2 \over \alpha + 1}. \]
The upper bound on the value of $m$ is included in the statement of the lemma here because it makes clear that, for any given positive integer $j$, it is possible to obtain an exhaustive list of all elements of $S_j$ (and hence a complete solution of~\eqref{eqn:generalized_Pell_equation}) by checking a finite number of pairs $(m,n)$, namely those which satisfy~\eqref{eqn:Pell_inequalities}, to see which of them satisfy~\eqref{eqn:generalized_Pell_equation}.

\noindent
(ii) It is apparent that every $(m,n) \in \mathbb{Z}^2$ satisfying the inequalities
\begin{equation} \label{eqn:fundamental_Pell_inequalities}
m \geq 1 \quad \text{and} \quad -{\beta \over \alpha+1}m < n \leq {\beta \over \alpha+1}m
\end{equation}
must belong to $S_j$ for some $j \geq 1$. 
Consequently, for any function $\phi$ for which the following sums converge absolutely, we may write
\[ 
\sum_{j=1}^\infty \sum_{(m,n) \in S_j} \phi(m,n) = \sum_{m=1}^\infty \sum_{n=1-\lceil {\beta m \over \alpha+1} \rceil}^{\lfloor {\beta m \over \alpha+1} \rfloor} \phi(m,n). 
\]
This observation is used in the proof of Theorem~\ref{thm:hol_proj_eval}.}
\end{remark}

\begin{example}{\normalfont
If $D=24$ and $j=25$, then the fundamental solution of~\eqref{eqn:Pell_equation} is $(5,1)$ and the set $S_j$ defined in Lemma~\ref{lem:Pell} consists of the three solutions $(5,0), (7,1)$, and $(7,-1)$. 
The complete set of solutions in $\mathbb{Z}^2$ to $m^2 - 24n^2 = 25$ is therefore generated by putting
\[ 
m+n\sqrt{24} = \pm5(5+\sqrt{24}\,)^k \text{ or } \pm(7\pm\sqrt{24}\,) (5+\sqrt{24}\,)^k 
\]
where, in each instance, $k$ may take any integer value and any combination of signs may be chosen.}
\end{example}

\begin{theorem} \label{thm:hol_proj_eval} 
Let $a$, $b$, $c$, $d$ be integers such that $abcd$ is not a square number. 
Let $\chi_1$ and $\chi_2$ be Dirichlet characters with moduli $p_1$ and $p_2$ respectively. 
Suppose that $p_1$ divides $bc$ and $p_2$ divides $ad$. 
Set
\[ 
g = \gcd \! \left( {bc \over p_1}, {ad \over p_2} \right).
\]
Let $(\alpha,\beta)$ be the fundamental solution of
\[ 
m^2-{abcd \over g^2}n^2=1,
\]
and assume that $(\alpha-1)/p_1$ and $(\alpha+1)/p_2$ are not both integers. 
Fix $\varepsilon_1, \varepsilon_2 \in \{1,-1\}$, and define
\begin{equation} \label{eqn:epsilon_values}
\varepsilon_3 = \left(\varepsilon_1^{bc \over g} \varepsilon_2\right)^{ad\beta \over g}, 
\qquad 
\varepsilon_4 = \left(\varepsilon_1 \varepsilon_2^{ad \over g} \right)^{bc\beta \over g}.
\end{equation}
Finally, let
\begin{align*}
f(\tau) = -i \left( \sum_{m=1}^\infty \varepsilon_1^m \chi_1(m) m \, e^{{2\pi i a m^2 \over b} \tau} \right) \sum_{n=1}^\infty \varepsilon_2^n \chi_2(n) n \int_{-\overline{\tau}}^{i \infty} {e^{{2\pi i c n^2 \over d}z} \over \sqrt{-i(z+\tau)}} \, \dd z.
\end{align*}
Then the weight-$2$ holomorphic projection of $f$ is given by
\begin{equation} \label{eqn:hol_proj_even}
\pi_\text{hol}(f)(\tau) = \sqrt{b \over 2a} \sum_{m=1}^\infty \sum_{n=1}^{\lfloor {ad\beta \over g(\alpha+1)}m \rfloor} \varepsilon_1^m \varepsilon_2^n \chi_1(m) \chi_2(n) \left( {\alpha+\chi_1(\alpha)\chi_2(\alpha) \varepsilon_3^m \varepsilon_4^n \over bc\beta/g} m - n \right) e^{2\pi i \tau ({a \over b} m^2-{c \over d}n^2)}
\end{equation}
if $\chi_2$ is an even character, and by
\begin{equation} \label{eqn:hol_proj_odd}
\pi_\text{hol}(f)(\tau) = \sqrt{d \over 2c} \sum_{m=1}^\infty \sum_{n=1}^{\lfloor {ad\beta \over g(\alpha+1)}m \rfloor} \varepsilon_1^m \varepsilon_2^n \chi_1(m) \chi_2(n) \left( m - {\alpha+\chi_1(\alpha)\chi_2(\alpha) \varepsilon_3^m \varepsilon_4^n \over ad\beta/g} n \right) e^{2\pi i \tau ({a \over b} m^2-{c \over d}n^2)}
\end{equation}
if $\chi_2$ is an odd character.
\end{theorem}

\begin{proof} 
First note that, by an elementary change of variables in the integral, the function $f$ may be written as
\[ 
f(\tau) = \sum_{m,n=1}^\infty mn \varepsilon_1^m \varepsilon_2^n\chi_1(m) \chi_2(n) e^{{2\pi i \tau \over bd}(adm^2-bcn^2)} \int_0^\infty {e^{-{2\pi c n^2 \over d} (x+2v)} \over \sqrt{x+2v}} \, \dd x, 
\]
where $v = \Im(\tau)$. 
By a well-known integral transformation formula, see~\cite[Section~1]{Yzeren1979}, this may be rewritten as
\[ 
f(\tau) = {2 \over \pi} \sum_{m,n=1}^\infty mn \varepsilon_1^m \varepsilon_2^n \chi_1(m) \chi_2(n) e^{{2\pi i \tau \over bd}(adm^2-bcn^2)} \int_0^\infty {e^{-2\pi v(x^2+{2cn^2 \over d})} \over x^2+{2cn^2 \over d}} \, \dd x. 
\]
Considering that $adm^2-bcn^2$ is necessarily a multiple of $g$ (since $g$ divides both $ad$ and $bc$), the Fourier series expansion of $f$ is thus
\[ 
f(\tau) = \sum_{j=-\infty}^\infty A_j(v) e^{2\pi ijg\tau \over bd}, 
\]
with the coefficients $A_j(v)$ being given by
\begin{equation} \label{eqn:Fourier_coefficients_A}
A_j(v) = {2 \over \pi} \sum_{\substack{m,n \geq 1 \\ adm^2-bcn^2 = gj}} mn \,\varepsilon_1^m \varepsilon_2^n \,\chi_1(m)\, \chi_2(n) \int_0^\infty {e^{-2\pi v(x^2+{2cn^2 \over d})} \over x^2+{2cn^2 \over d}} \, \dd x.
\end{equation}
In accordance with Definition~\ref{def:IRRdef}, the weight-2 holomorphic projection of $f$ is therefore given by
\begin{equation} \label{eqn:initial_hol_proj}
\pi_\text{hol}(f)(\tau) = C + \sum_{j=1}^\infty B_j e^{2\pi ijg\tau \over bd},
\end{equation}
where
\[ 
C = \lim_{\tau \rightarrow i\infty} f(\tau) \qquad \text{and} \qquad B_j = {4gj\pi \over bd} \int_0^\infty A_j(v) e^{-{4gj\pi v \over bd}} \, \dd v. 
\]
It is straightforward to see that $C=0$. 
By substituting the formula~\eqref{eqn:Fourier_coefficients_A} for $A_j(v)$ into the formula for $B_j$, we see that these coefficients $B_j$ are given by
\[ 
B_j = {8gj \over bd} \sum_{\substack{m,n \geq 1 \\ adm^2-bcn^2 = gj}} mn \,\varepsilon_1^m \varepsilon_2^n \,\chi_1(m)\, \chi_2(n) \int_0^\infty \int_0^\infty {e^{-2\pi v(x^2+{2cn^2 \over d}+{2gj \over bd})} \over x^2+{2cn^2 \over d}} \, \dd v \, \dd x. 
\]
The integral over $v$ is straightforward to compute, and the double integral in this formula is then seen to be equal to
\begin{align*}
{1 \over 2\pi}\int_0^\infty {\dd x \over (x^2+{2am^2 \over b})(x^2+{2cn^2 \over d})} &= {1 \over 8mn\sqrt{2ac/bd}\, (m\sqrt{a/b}+n\sqrt{c/d})} \\
&= {bd \over 8gjmn} \sqrt{bd \over 2ac} \left(m\sqrt{a \over b}-n\sqrt{c \over d} \right),
\end{align*}
where we have used the fact $adm^2-bcn^2 = gj$ twice to simplify the expression.
Hence, the coefficients $B_j$ are given by
\begin{equation} \label{eqn:Fourier_coefficients_B}
B_j = \sqrt{bd \over 2ac} \sum_{\substack{m,n \geq 1 \\ adm^2-bcn^2 = gj}} \varepsilon_1^m \varepsilon_2^n \chi_1(m) \chi_2(n) \left(m\sqrt{a \over b}-n\sqrt{c \over d} \right).
\end{equation}

Our aim now is to evaluate the infinite series in this formula and thus determine, for any given $j \geq 1$, a closed-form expression for the coefficient $B_j$. 
In order to do this, let us begin by putting
\[ 
\lambda = {ad \over g}, \qquad \mu = {bc \over g},
\]
and making the substitution $m \mapsto m/\lambda$ in~\eqref{eqn:Fourier_coefficients_B} to obtain
\begin{equation} \label{eqn:Fourier_coefficients_B_ver_2}
B_j = {g \over a\sqrt{2cd}} \sum_{\substack{m,n \geq 1 \\ m^2-\lambda\mu n^2 = \lambda j \\ m \equiv 0 \text{ (mod $\lambda$)}}} \varepsilon_1^{m/\lambda} \varepsilon_2^n\, \chi_1(m/\lambda)\, \chi_2(n) \left(m-n\sqrt{\lambda\mu} \right).
\end{equation}
Next, let us consider the equation
\begin{equation} \label{eqn:generalized_Pell}
m^2-\lambda\mu n^2 = \lambda j,
\end{equation}
which is an instance of the generalized Pell equation~\eqref{eqn:generalized_Pell_equation} with $D=\lambda\mu$ and $j \mapsto \lambda j$. 
In keeping with the notation used in Lemma~\ref{lem:Pell}, let us write $S_{\lambda j}$ for the set of all pairs $(m,n) \in \mathbb{Z}^2$ satisfying the inequalities~\eqref{eqn:fundamental_Pell_inequalities} as well as the requirement that $(m,n)$ be a solution of~\eqref{eqn:generalized_Pell}. 
The complete set of solutions of~\eqref{eqn:generalized_Pell} is then given by $\{(u_k,v_k)\}$, where $k$ runs through all integer values, $(u_0,v_0)$ runs through all elements of $\pm S_{\lambda j}$, and $u_k$ and $v_k$ denote the integers determined by
\begin{equation} \label{eqn:Pell_equivalence_class_generation}
u_k + v_k \sqrt{\lambda\mu} = (u_0 + v_0 \sqrt{\lambda\mu} \,) (\alpha + \beta \sqrt{\lambda\mu} \,)^k.
\end{equation}
We may therefore express the formula~\eqref{eqn:Fourier_coefficients_B_ver_2} in the alternative form
\begin{equation} \label{eqn:Fourier_coefficients_B_ver_3}
B_j = {g \over a\sqrt{2cd}} \sum_{(u_0,v_0) \in \pm S_{\lambda j}} \sum_{\substack{k \in \mathbb{Z} \\ u_k \equiv 0 \text{ (mod $\lambda$)} \\ u_k, v_k \geq 1}} \varepsilon_1^{u_k/\lambda} \varepsilon_2^{v_k}\, \chi_1(u_k/\lambda)\, \chi_2(v_k) \left(u_k-v_k\sqrt{\lambda\mu} \right),
\end{equation}
where it is to be understood that the $u_k$ and $v_k$ are the functions of $u_0$ and $v_0$ defined by~\eqref{eqn:Pell_equivalence_class_generation}. 
In this formula, the sum over $k$ is very nearly a geometric series since, by~\eqref{eqn:Pell_equivalence_class_generation}, the factor of $u_k-v_k\sqrt{\lambda\mu}$ in the summand is equal to $(u_0 - v_0 \sqrt{\lambda\mu} \,) (\alpha - \beta \sqrt{\lambda\mu} \,)^k$. 
The only obstacles to summing the series in this way are the restrictions on $u_k$ and $v_k$ beneath the summation symbol and the dependence on $k$ of the factors involving the Dirichlet characters $\chi_1$ and $\chi_2$ and the powers of $\varepsilon_1$ and $\varepsilon_2$. 
These obstacles may be overcome as follows.

Let us first address the question of when the value of $u_k$ is positive. 
The formula~\eqref{eqn:Pell_equivalence_class_generation} implies that the terms of the sequence $(u_k)$ are given explicitly by
\[ u_k = {1 \over 2} \left( (u_0 + v_0 \sqrt{\lambda\mu} \,)(\alpha + \beta \sqrt{\lambda\mu} \,)^k + (u_0 - v_0 \sqrt{\lambda\mu} \,)(\alpha - \beta \sqrt{\lambda\mu} \,)^k \right). \]
Since $|v_0/u_0| \leq \beta/(\alpha+1)$, it follows that
\[ {u_k \over u_0} \geq {\alpha - \beta\sqrt{\lambda\mu} + 1 \over 2(\alpha+1)} (\alpha + \beta \sqrt{\lambda\mu} \,)^{|k|} + {\alpha + \beta\sqrt{\lambda\mu} + 1 \over 2(\alpha+1)} (\alpha - \beta \sqrt{\lambda\mu} \,)^{|k|}. \]
Since $\alpha > \beta\sqrt{\lambda\mu}$, the last expression is necessarily positive. Hence, all terms of the sequence $(u_k)$ are of the same sign. 
In order that $u_k \geq 1$ should hold for a given integer $k$ in the inner summation of~\eqref{eqn:Fourier_coefficients_B_ver_3}, it is therefore necessary and sufficient that we choose the plus sign on the outer summation. 
We may thus rewrite~\eqref{eqn:Fourier_coefficients_B_ver_3} as
\begin{equation} \label{eqn:Fourier_coefficients_B_ver_4}
B_j = {g \over a\sqrt{2cd}} \sum_{(u_0,v_0) \in S_{\lambda j}} \sum_{\substack{k \in \mathbb{Z} \\ u_k \equiv 0 \text{ (mod $\lambda$)} \\ v_k \geq 1}} \varepsilon_1^{u_k /\lambda} \varepsilon_2^{v_k}\, \chi_1(u_k/\lambda)\, \chi_2(v_k) \left(u_k-v_k\sqrt{\lambda\mu} \right).
\end{equation}

We next consider the requirement that $v_k$ be positive for every $k$ in the inner sum. By~\eqref{eqn:Pell_equivalence_class_generation}, the terms of the sequence $(v_k)$ are given by 
\[ v_k = {1 \over 2\sqrt{\lambda\mu}} \left( (u_0 + v_0 \sqrt{\lambda\mu} \,)(\alpha + \beta \sqrt{\lambda\mu} \,)^k - (u_0 - v_0 \sqrt{\lambda\mu} \,)(\alpha - \beta \sqrt{\lambda\mu} \,)^k \right). \]
From the same inequality for $|v_0/u_0|$ as was used previously, it follows readily that
\[ 
\sgn(k) {v_k \over u_0} \geq {\alpha - \beta\sqrt{\lambda\mu} + 1 \over 2(\alpha+1)} (\alpha + \beta \sqrt{\lambda\mu} \,)^{|k|} - {\alpha + \beta\sqrt{\lambda\mu} + 1 \over 2(\alpha+1)} (\alpha - \beta \sqrt{\lambda\mu} \,)^{|k|}. 
\]
The right-hand side of this inequality is positive for every $k \neq 0$. This assertion may be verified as follows: elementary rearrangements show that it is equivalent to
\[ 
(\alpha+\beta\sqrt{\lambda\mu}\,)^{2k} > {(\alpha+\beta\sqrt{\lambda\mu}+1)^2 \over 2(\alpha+1)}, \quad \text{for every $k \geq 1$.} 
\]
In this inequality, the left-hand side is an increasing function of $k$, so it suffices to check the case $k=1$. 
This is equivalent to
\[ 
\alpha+\beta\sqrt{\lambda\mu} > {1 \over \sqrt{2(\alpha+1)}-1}, 
\]
which clearly must hold since (considering that $\alpha\geq2$ and $\beta\geq1$) the left-hand side is at least $2+\sqrt{\lambda\mu}$ while the right-hand side is at most $(1+\sqrt{6})/5$. 
Hence, $v_k$ is positive for every $k \geq 1$ and negative for every $k \leq -1$. The value of $v_0$ may, however, be positive or negative or zero according to which element of $S_j$ is chosen in the outer summation of~\eqref{eqn:Fourier_coefficients_B_ver_4}. 
The condition $v_k \geq 1$ on the inner summation is therefore equivalent to $k \geq 1-\delta(v_0)$, where $\delta$ is defined by
\[ 
\delta(n) = \begin{cases} 1 \qquad \text{if $n \geq 1$,} \\ 0 \qquad \text{otherwise.} \end{cases} 
\]

From~\eqref{eqn:Pell_equivalence_class_generation}, we obtain the recurrence relations
\begin{equation} \label{eqn:uv_recurrence}
u_{k+1} = \alpha u_k + \beta\lambda\mu v_k \qquad \text{and} \qquad v_{k+1} = \beta u_k + \alpha v_k.
\end{equation}
From the first of these, it follows that $u_{k+1} \equiv \alpha u_k$ (mod $\lambda$) and consequently $u_k \equiv \alpha^k u_0$ (mod $\lambda$) for each $k$. 
Since $\alpha^2 = \lambda\mu \beta^2 + 1$, the integers $\alpha$ and $\lambda$ are coprime to one another, and hence $\lambda$ divides $u_k$ if and only if $\lambda$ divides $u_0$. 
We may therefore replace $u_k$ by $u_0$ in the condition $u_k \equiv 0$(mod $\lambda$) appearing in the inner summation of~\eqref{eqn:Fourier_coefficients_B_ver_4}. 
This condition is thus independent of $k$ and may be moved to the outer summation.

Next, suppose that $u_0$ is indeed divisible by $\lambda$. 
Then, by the argument just given, each term of the sequence $(u_k)_{k \geq 0}$ is divisible by $\lambda$. Since $p_1$ divides $\mu$ and $p_2$ divides $\lambda$, we see that
\[ 
\chi_1\left(u_{k+1} \over \lambda\right) = \chi_1\left({\alpha u_k \over \lambda} + \beta\mu v_k \right) = \chi_1\left({\alpha u_k \over \lambda} \right) \qquad \text{and} \qquad \chi_2(v_{k+1}) = \chi_2(\beta u_k + \alpha v_k) = \chi_2(\alpha v_k). 
\]
Iterating these identities gives
\[
\chi_1(u_k/\lambda) = \chi_1(\alpha)^k \chi_1(u_0/\lambda) \qquad \text{and} \qquad \chi_2(v_k) = \chi_2(\alpha)^k \chi_2(v_0)
\]
for every $k$. Since $\chi_1(\alpha^2) = \chi_1(\lambda\mu\beta^2+1)=\chi_1(1)=1$ and similarly $\chi_2(\alpha^2)=1$, we have $\chi_1(\alpha), \chi_2(\alpha)\in\{-1,1\}$. 
This establishes that
\[ 
\chi_1(u_k/\lambda) \, \chi_2(v_k) = \begin{cases} \chi_1(u_0/\lambda)\,\chi_2(v_0) \qquad &\text{if $k$ is even,} \\ \chi_1(\alpha)\,\chi_2(\alpha) \,\chi_1(u_0/\lambda)\,\chi_2(v_0) \qquad &\text{if $k$ is odd.} \end{cases} 
\]

In order to put the series in~\eqref{eqn:Fourier_coefficients_B_ver_4} in a form that may readily be evaluated, it remains only to determine how the expression $\varepsilon_1^{u_k/\lambda} \varepsilon_2^{v_k}$ in the summand depends on $u_0$ and $v_0$. 
This is straightforward since these powers of $\varepsilon_1$ and $\varepsilon_2$ depend only on the values of their exponents modulo $2$. 
The relations in~\eqref{eqn:uv_recurrence} imply that
\[ u_{k+2} - 2\alpha u_{k+1} + u_k = 0, \qquad v_{k+2} - 2\alpha v_{k+1} + v_k = 0, \]
and hence (supposing as before that every $u_k$ is divisible by $\lambda$), it follows that $u_{k+2}/\lambda \equiv u_k/\lambda$ (mod~$2$) and $v_{k+2} \equiv v_k$ (mod~$2$) for every $k$. 
Thus,
\[ 
{u_k \over \lambda} \equiv \begin{cases} {u_0 \over \lambda} \quad &\text{if $k$ is even,} \\ {\alpha u_0 \over \lambda} + \beta \mu v_0 \quad &\text{if $k$ is odd} \end{cases} \quad \text{(mod $2$)} 
\]
and
\[ 
v_k \equiv \begin{cases} v_0\quad &\text{if $k$ is even,} \\ \beta u_0 + \alpha v_0 \quad &\text{if $k$ is odd} \end{cases} \quad \text{(mod $2$).} 
\]

Hence ~\eqref{eqn:Fourier_coefficients_B_ver_4} becomes
{\fontsize{9}{10}\selectfont
\begin{align*}
B_j &= {g \over a\sqrt{2cd}} \sum_{\substack{(u_0,v_0) \in S_j \\ u_0 \equiv 0 \text{ (mod $\lambda$)}}} \left( \varepsilon_1^{u_0 \over \lambda} \varepsilon_2^{v_0} \sum_{k=1-\delta(v_0)}^\infty (\alpha - \beta \sqrt{\lambda\mu} \,)^{2k} + \varepsilon_1^{{\alpha u_0 \over \lambda}+\beta\mu v_0} \varepsilon_2^{\beta u_0 + \alpha v_0} \chi_1(\alpha)\chi_2(\alpha) \sum_{k=0}^\infty (\alpha - \beta \sqrt{\lambda\mu} \,)^{2k+1} \right) \\ &\quad \hphantom{{g \over a\sqrt{2cd}} \sum_{\substack{(u_0,v_0) \in S_j \\ u_0 \equiv 0 \text{ (mod $\lambda$)}}}} \times \chi_1(u_0/\lambda) \chi_2(v_0) (u_0 - v_0 \sqrt{\lambda\mu} \,) \\
&= {g \over a\sqrt{2cd}} \cdot {1 \over 2\beta\sqrt{\lambda\mu}} \sum_{\substack{(u_0,v_0) \in S_j \\ u_0 \equiv 0 \text{ (mod $\lambda$)}}} \left( (\alpha + \beta \sqrt{\lambda\mu} \,)^{1-2\delta(v_0)} + \chi_1(\alpha)\chi_2(\alpha) \varepsilon_1^{{(\alpha-1) u_0 \over \lambda}+\beta\mu v_0} \varepsilon_2^{\beta u_0 + (\alpha-1) v_0} \right) \\
&\quad \hphantom{{g \over a\sqrt{2cd}} \cdot {1 \over 2\beta\sqrt{\lambda\mu}} \sum_{\substack{(u_0,v_0) \in S_j \\ u_0 \equiv 0 \text{ (mod $\lambda$)}}}} \times \varepsilon_1^{u_0 \over \lambda} \varepsilon_2^{v_0} \chi_1(u_0/\lambda) \chi_2(v_0) (u_0 - v_0 \sqrt{\lambda\mu} \,).
\end{align*}
}
Finally, this is the required expression for $B_j$ in finite terms (since $S_j$ contains only finitely many elements). 
Substituting it into~\eqref{eqn:initial_hol_proj} and using the observation in Remark~\ref{rmk:Pell}(ii), we obtain
{\fontsize{11}{10}\selectfont
\begin{align*}
\pi_\text{hol}(f)(\tau) = {g^2 \over 2a\beta cd\sqrt{2ab}} \sum_{\substack{m=1 \\ m \equiv 0 \text{ (mod $\lambda$)}}}^\infty \sum_{n=1-\lceil {\beta m \over \alpha+1} \rceil}^{\lfloor {\beta m \over \alpha+1} \rfloor} &\left( (\alpha - \beta \sqrt{\lambda\mu} \,)^{1-2\delta(n)} + \chi_1(\alpha)\chi_2(\alpha) \varepsilon_3^{m \over \lambda} \varepsilon_4^n \right) \\
&\times \varepsilon_1^{m \over \lambda} \varepsilon_2^n \chi_1(m/\lambda) \chi_2(n) (m - n \sqrt{\lambda\mu} \,) e^{2\pi i g(m^2-\lambda\mu n^2)\tau \over bd\lambda}.
\end{align*}
}
Here, the fact that
\[ \varepsilon_1^{{(\alpha-1) m \over \lambda}+\beta\mu n} \varepsilon_2^{\beta m + (\alpha-1) n} = \varepsilon_1^{\beta \mu(m+n)} \varepsilon_2^{\beta(m + \lambda \mu n)} = \varepsilon_3^{m \over \lambda} \varepsilon_4^n \]
has been used ($\alpha-1$ being congruent to $\alpha^2-1\equiv\beta^2\lambda\mu$ (mod $2$), and $m/\lambda$ being an integer). 
After changing $m \mapsto \lambda m$, we may write $\pi_\text{hol}(f)$ as a sum of two series, the first containing all terms with $n \geq 1$ and the second containing all those with $n \leq 0$. 
After making the substitution $n \mapsto -n$ in the second of these two series, the holomorphic projection of $f$ is expressed as
{\fontsize{9}{10}\selectfont
\begin{align}
&\pi_\text{hol}(f)(\tau) = \\
&{g^2\sqrt{\lambda} \over 2a\beta cd\sqrt{2ab}} \sum_{m=1}^\infty \sum_{n=1}^{\lfloor {\beta \lambda m \over \alpha+1} \rfloor} \left( \alpha + \beta \sqrt{\lambda\mu} + \chi_1(\alpha)\chi_2(\alpha) \varepsilon_3^m \varepsilon_4^n \right) \varepsilon_1^m \varepsilon_2^n \chi_1(m) \chi_2(n) (m\sqrt{\lambda} - n \sqrt{\mu} \,) e^{2\pi i \tau ({a \over b} m^2-{c \over d}n^2)} \nonumber \\
&+ {g^2 \chi_2(-1) \sqrt{\lambda} \over 2a\beta cd\sqrt{2ab}} \sum_{m=1}^\infty \sum_{n=0}^{\lceil {\beta \lambda m \over \alpha+1} \rceil-1} \left( \alpha - \beta \sqrt{\lambda\mu} + \chi_1(\alpha)\chi_2(\alpha) \varepsilon_3^m \varepsilon_4^n \right) \varepsilon_1^m \varepsilon_2^n  \chi_1(m) \chi_2(n) (m\sqrt{\lambda} + n \sqrt{\mu} \,) e^{2\pi i \tau ({a \over b} m^2-{c \over d}n^2)}. \label{eqn:hol_proj_rewritten}
\end{align}
}

There is now the question of whether the two series may be combined into one in a simple way. 
This is indeed possible: the summation $\sum_{n=0}^{\lceil {\beta \lambda m \over \alpha+1} \rceil-1}$ may be replaced by $\sum_{n=1}^{\lfloor {\beta \lambda m \over \alpha+1} \rfloor}$ in the second series without altering the value of the expression. 
The change in the lower limit causes no difficulty, since the term with $n=0$ makes no contribution to the sum because $\chi_2(0)=0$. 
The upper limits agree unless $\beta\lambda m/(\alpha+1)$ is an integer. 
It therefore remains to show that, in this case, $\chi_1(m)\chi_2(n) = 0$ when $n = \beta\lambda m/(\alpha+1)$. 
To this end, let us write $\beta\lambda/(\alpha+1) = r/s$ in lowest terms. 
Then the values of $m$ and $n$ in question are necessarily such that $m$ is divisible by $s$ and $n$ is divisible by $r$. Therefore, it suffices to show that $\chi_1(s)\chi_2(r)=0$. 
By assumption (from the statement of the theorem), $(\alpha-1)/p_1$ and $(\alpha+1)/p_2$ are not both integers. 
If we assume first that $\alpha+1$ is not divisible by $p_2$, then we may choose some $\nu>1$ such that $\nu$ divides $p_2$ and $\nu$ is coprime to $\alpha+1$. 
Since $p_2$ divides $\lambda$, we have $\nu$ divides $r$, and it follows that $\chi_2(r)=0$. 
If, on the other hand, we assume that $\alpha-1$ is not divisible by $p_1$, then we may choose $\nu>1$ such that $\nu$ divides $p_1$ and $\nu$ is coprime to $\alpha-1$. 
From writing
\[ 
{r \over s} = {\beta\lambda \over \alpha+1} = {\alpha-1 \over \beta\mu}, 
\]
we see, since $p_1$ divides $\mu$, that $\nu$ divides $s$ and so $\chi_1(s)=0$.

Hence, we may write~\eqref{eqn:hol_proj_rewritten} as
{\fontsize{10.2}{10}\selectfont
\begin{align*}
\pi_\text{hol}(f)(\tau) &= {g^2\sqrt{\lambda} \over a\beta cd\sqrt{2ab}} \sum_{m=1}^\infty \sum_{n=1}^{\lfloor {\beta \lambda m \over \alpha+1} \rfloor} \varepsilon_1^m \varepsilon_2^n \chi_1(m) \chi_2(n) \bigg( \left( \alpha + \beta \sqrt{\lambda\mu} + \chi_1(\alpha)\chi_2(\alpha)  \varepsilon_3^m \varepsilon_4^n \right) (m\sqrt{\lambda}- n \sqrt{\mu} \,) \\
&\quad + \chi_2(-1) \left( \alpha - \beta \sqrt{\lambda\mu} + \chi_1(\alpha)\chi_2(\alpha) \varepsilon_3^m \varepsilon_4^n \right) (m\sqrt{\lambda} + n \sqrt{\mu} \,) \bigg) e^{2\pi i \tau ({a \over b} m^2-{c \over d}n^2)}.
\end{align*}
}
This simplifies down to either~\eqref{eqn:hol_proj_even} or~\eqref{eqn:hol_proj_odd} according to whether $\chi_2(-1)$ is $1$ or $-1$ respectively.
\end{proof}

\begin{remark}
If $(\alpha-1)/p_1$ and $(\alpha+1)/p_2$ are both integers, but all other conditions in the statement of Theorem~\ref{thm:hol_proj_eval} hold, then the formula~\eqref{eqn:hol_proj_rewritten} may still be used to determine the weight-2 holomorphic projection of $f$.

For example, if $f$ is given by
\begin{equation} \label{eqn:hol_proj_example_for_conditions} 
f(\tau) = -i \left( \sum_{m=1}^\infty \left( {m \over 7} \right) m \, e^{{16\pi i m^2 \over 7} \tau} \right) \sum_{\substack{n=1 \\ n \text{ odd}}}^\infty n \int_{-\overline{\tau}}^{i \infty} {e^{{i\pi n^2}z} \over \sqrt{-i(z+\tau)}} \, \dd z,
\end{equation}
then we may take $a=8$, $b=7$, $c=1$, $d=2$. 
In the sum over $m$, the Dirichlet character is $\chi_1(m) = (m/7)$, of modulus $p_1=7$. 
In the sum over $n$, the Dirichlet character is the principal character modulo $2$, so we take $p_2=2$. 
The value of $g$ is $\gcd(8,1)=1$, and the corresponding Pell equation is $m^2-112n^2=1$. 
Its fundamental solution is $(m,n) = (127,12)$, so we put $\alpha=127$ and $\beta=12$. Then $(\alpha-1)/p_1 = 18$ and $(\alpha+1)/p_2 = 64$ are both integers. From~\eqref{eqn:hol_proj_rewritten}, we find that the holomorphic projection of $f$ is
\[ 
{1 \over 12\sqrt{7}} \sum_{m=1}^\infty \sum_{\substack{n=1 \\ n \text{ odd}}}^{\lfloor {3m \over 2} \rfloor} \left( {m \over 7} \right) (32m-21n) e^{i\pi\tau({16m^2 \over 7} - n^2)} - {1 \over 24\sqrt{7}} \sum_{\substack{n=1 \\ n \text{ odd}}}^\infty \left( {n \over 7} \right) n e^{\tfrac{i\pi n^2}{7}\tau }. 
\]
The first term in this expression is the same as that obtained from~\eqref{eqn:hol_proj_even}. 
The second term is a correction term arising from values of $m$ for which the upper limits of the two inner summations are not equal to one another and the final term of the second of those inner summations makes a non-zero contribution to the double sum. 
This can only occur when $(\alpha-1)/p_1$ and $(\alpha+1)/p_2$ are both integers.
\end{remark}

\begin{remark}
Since any periodic sequence $(u_n)$ can be written as a linear combination of Dirichlet characters, provided that $u_n=0$ whenever $n$ is not coprime to the period, we could determine the holomorphic projection of any function $f$ of the form
\[ 
\left( \sum_{m=1}^\infty u_m m \, e^{{2\pi i a m^2 \over b} \tau} \right) \sum_{n=1}^\infty v_n n \int_{-\overline{\tau}}^{i \infty} {e^{{2\pi i c n^2 \over d}z} \over \sqrt{-i(z+\tau)}} \, \dd z, 
\]
where $(u_n)$ and $(v_n)$ are two such sequences, using only the case $\varepsilon_1=\varepsilon_2=1$ of Theorem~\ref{thm:hol_proj_eval}. 
However, the resulting double series for $\pi_\text{hol}(f)$ obtained in this way is not necessarily the simplest possible. 
This occurs, for example, in connection with~\eqref{eqn:nu3}. To obtain this identity using Theorem~\ref{thm:hol_proj_eval}, we must determine the holomorphic projection of
\begin{equation} \label{eqn:hol_proj_example_for_nu}
{4i \over \sqrt{6}}\bigg( \sum_{m=1}^\infty \left({-4 \over m}\right) m e^{\tfrac{i\pi m^2 }{2}\tau} \bigg) \sum_{n=1}^\infty (-1)^{n-1} \left({n \over 3}\right) n \int_{-\bar{\tau}}^{i\infty} {e^{\tfrac{2\pi i n^2 }{3}z} \over \sqrt{-i(z+\tau)}} \, \dd z.
\end{equation}
We may write the expression $(-1)^{n-1} (n/3)$ as a linear combination of Dirichlet characters as $2(n/12) - (n/3)$. 
If now we make two applications of Theorem~\ref{thm:hol_proj_eval} with $\varepsilon_1=\varepsilon_2=1$, once with $\chi_2(n) = (n/12)$ and again with $\chi_2(n) = (n/3)$, we may combine the results to see that the holomorphic projection of~\eqref{eqn:hol_proj_example_for_nu} is
\[ \sum_{m=1}^\infty\sum_{n=1}^{\lfloor \frac{3}{4}m \rfloor} \left(\frac{-4}{m}\right)\left(\frac{n}{12}\right)\left(m-n\right)q^{\frac{m^2}{2}-\frac{2n^2}{3}-\frac{5}{6}} +\sum_{m=1}^\infty\sum_{n=1}^{\lfloor \frac{3}{8}m \rfloor} \left(\frac{-4}{m}\right)\left(\frac{n}{3}\right)\left(m-2n\right)q^{\frac{m^2}{2}-\frac{8n^2}{3}-\frac{5}{6}}. \]
However, if instead we take $\varepsilon_1=1$, $\varepsilon_2=-1$ and $\chi_2(n) = (n/3)$, we obtain the much simpler double series in~\eqref{eqn:nu3} for the holomorphic projection of~\eqref{eqn:hol_proj_example_for_nu}; see Section~\ref{sec:proofmain}. 
This is why the parameters $\varepsilon_1$ and $\varepsilon_2$ are included in Theorem~\ref{thm:hol_proj_eval}.
\end{remark}

The following lemma will be used in Section~\ref{sec:proofmain} and Appendix~\ref{app: identities} in order to simplify several identities.

\begin{lemma} \label{lem:double_ser_trans}
For any function $\varphi:\mathbb{Z}^2 \rightarrow \mathbb{C}$ such that the series below converge absolutely, the following identities hold:
\begin{align}
\sum_{m=1}^\infty \sum_{n=1}^{\lfloor {12 \over 7}m \rfloor} \left({-4 \over m}\right) \left({n \over 6}\right) \varphi(m,n) &= \sum_{m=1}^\infty \sum_{n=1}^{\lfloor {3 \over 2}m \rfloor} \left({-4 \over m}\right) \left({n \over 6}\right) \Big( \varphi(m,n) - \varphi(7m-4n,12m-7n) \Big), \label{eqn:double_ser_trans1} \\
\sum_{m=1}^\infty \sum_{n=1}^{\lfloor {24 \over 7}m \rfloor} \left({-4 \over m}\right) \left({n \over 6}\right) \varphi(m,n) &= \sum_{m=1}^\infty \sum_{n=1}^{3m} \left({-4 \over m}\right) \left({n \over 6}\right) \Big( \varphi(m,n) - \varphi(7m-2n,24m-7n) \Big). \label{eqn:double_ser_trans2}
\end{align}
\end{lemma}

\begin{proof}
Let $S$ denote the double series on the left-hand side of~\eqref{eqn:double_ser_trans1}. 
The unimodular substitution
\begin{equation} \label{eqn:unimodular_sub}
(m,n) \mapsto (2m-n, 3m-2n)
\end{equation}
changes the inequalities
\[ n \geq 0 \qquad \text{and} \qquad n \leq {7m \over 12}, \]
which define the range of the summation on the left-hand side of~\eqref{eqn:double_ser_trans1}, to
\[ n \leq {3m \over 2} \qquad \text{and} \qquad n \geq -{3m \over 2} \]
respectively. 
(For both summations we may take the lower limit to be zero since terms with $m=0$ or $n=0$ make no contribution to the sum.) 
It follows that
\[ S = \sum_{m=0}^\infty \sum_{n=-\lfloor {3m \over 2} \rfloor}^{\lfloor {3m \over 2} \rfloor} \left( {-4 \over 2m-n} \right) \left( {3m-2n \over 6} \right) \varphi(2m-n,3m-2n). \]
It is straightforward to check that the product of Kronecker symbols in the summand,
\[ \left( {-4 \over 2m-n} \right) \left( {3m-2n \over 6} \right), \]
is an odd function of $n$. 
The result of pairing terms of index $n$ in the inner sum with those of index $-n$ is therefore
\begin{equation} \label{eqn:S_rewritten}
S = \sum_{m=1}^\infty \sum_{n=1}^{\lfloor {3m \over 2} \rfloor} \left( {-4 \over 2m-n} \right) \left( {3m-2n \over 6} \right) \Big( \varphi(2m-n,3m-2n) - \varphi(2m+n,3m+2n) \Big).
\end{equation}
Now make the substitution~\eqref{eqn:unimodular_sub} once again. 
The inequalities
\[ n \geq 0 \qquad \text{and} \qquad n \leq {3m \over 2}, \]
which define the range of the summation in~\eqref{eqn:S_rewritten}, become
\[ n \leq {3m \over 2} \qquad \text{and} \qquad n \geq 0 \]
respectively. 
Hence, the right-hand side of~\eqref{eqn:S_rewritten} is identical to that of~\eqref{eqn:double_ser_trans1}.

Now, in order to establish the identity~\eqref{eqn:double_ser_trans2}, let
\[ T= \sum_{m=1}^\infty \sum_{n=3m+1}^{\lfloor {24 \over 7}m \rfloor} \left({-4 \over m}\right) \left({n \over 6}\right) \varphi(m,n). \]
The unimodular substitution $(m,n) \mapsto (7m-2n,24m-7n)$ changes the inequalities
\[ n > 3m \qquad \text{and} \qquad n \leq {24m \over 7} \]
to
\[ n < 3m \qquad \text{and} \qquad n \geq 0 \]
respectively. 
The result of applying this substitution to the series $T$ is therefore
\[ T = \sum_{m=1}^\infty \sum_{n=0}^{3m-1} \left({-4 \over 7m-2n}\right) \left({24m-7n \over 6}\right) \varphi(7m-2n,24m-7n). \]
It is straightforward to verify that, for any integers $m$ and $n$,
\[ \left({-4 \over 7m-2n}\right) \left({24m-7n \over 6}\right) = -\left({-4 \over m}\right) \left({n \over 6}\right). \]
Consequently,
\begin{equation} \label{eqn:T_rewritten}
T = -\sum_{m=1}^\infty \sum_{n=1}^{3m} \left({-4 \over m}\right) \left({n \over 6}\right) \varphi(7m-2n,24m-7n),
\end{equation}
the change of limits in the inner summation being justified since $({n\over6})$ is zero when $n=0$ or $3m$. 
The identity~\eqref{eqn:double_ser_trans2} now follows from adding the series
\[ \sum_{m=1}^\infty \sum_{n=1}^{3m} \left({-4 \over m}\right) \left({n \over 6}\right) \varphi(m,n) \]
to both sides of~\eqref{eqn:T_rewritten}.
\end{proof}

%Proof of the main result---------------------------------------------------------------------------------------------------------------------------------------------------------------------------------------------------------------------------------------------------------------------------------------
\section{\textbf{Proof of Theorem \ref{thm: main}}}
\label{sec:proofmain}
We derive eighteen convolution-type $q$-series identities from the eighteen holomorphic projection identities obtained in Theorem \ref{thm: equality}; see Appendix \ref{app: identities}. 
The identities in the Appendix are listed according to the order of the components in the Theorem. 
Corollary \ref{cor:holoprojids} records the identities for $\holoproj{H_{k} \, N_j}$, where $0 \le k \le 5$ and $0 \le j \le2$. 
The convolution identities may be obtained by applying the method of Section \ref{sec:formulaholoproj} to calculate each  $\holoproj{H_{k}\, N_j}$. 
The five identities stated in Theorem \ref{thm: main} involve $\psi$ and $\nu$; it
can be shown that the remaining identities are equivalent to one of these, 
using \eqref{eqn:phi-psi_relation} and \eqref{eqn:nu-omega_relation}.
Some detail of this is given in Appendix~\ref{sec:otherids}.
In this section, we illustrate the derivation of each of
the five identities in Theorem \ref{thm: main}. 

We derive the identities~\eqref{eqn:psi2}--\eqref{eqn:nu3} from the holomorphic projection identities for
\begin{align}
\holoproj{H_{2}N_1},\qquad \holoproj{H_{2}N_2},\qquad \holoproj{H_{3}N_1}, 
\qquad \holoproj{H_{4}N_1},\qquad \holoproj{H_{4}N_2},
\end{align}
respectively in Corollary~\ref{cor:holoprojids}.
\begin{proof}[Proof of Theorem \ref{thm: main}]
By definition in Theorem~\ref{thm:KK},
\begin{equation*}
    H_{(3)} \otimes N = F_{(3)}\otimes N + G_{(3)}\otimes N.
\end{equation*}
Since $F_{(3)}\otimes N$ is already holomorphic, we obtain using \eqref{eqn:split hol_proj}
\begin{equation*}
\holoproj{H_{(3)} \otimes N} = F_{(3)}\otimes N + \holoproj{G_{(3)}\otimes N}.
\end{equation*}
Using~\eqref{eqn:F_(3)_def},~\eqref{eq: H}, and writing $G_{(3)}\eqqcolon(G_0,\ldots,G_5)^T$, this implies
\begin{equation}\label{eqn:holproj_holproj}
\begin{aligned}
&
\holoproj{H_{2}N_1}(\tau) = q^{-1/48}\eta^3(\tau/2)\psi(q^{1/2}) + \frac{1}{2}\holoproj{\eta^{3}(\tau/2)G_{2}(\tau)},\\
&
\holoproj{H_{2}N_2}(\tau) = 2\sqrt{2}\,q^{-1/48}\eta^3(2\tau)\psi(q^{1/2}) + \sqrt{2}\,\holoproj{\eta^{3}(2\tau)G_{2}(\tau)},\\
&
\holoproj{H_{3}N_1}(\tau) = q^{-1/48}\eta^3(\tau/2)\psi(-q^{1/2}) + \frac{1}{2}\holoproj{\eta^{3}(\tau/2)G_{3}(\tau)},\\
&
\holoproj{H_{4}N_1}(\tau) = \frac{1}{\sqrt{2}}\, q^{1/6}\eta^3(\tau/2)\nu(q^{1/2}) + \frac{1}{2}\holoproj{\eta^{3}(\tau/2)G_{4}(\tau)},\\
&
\holoproj{H_{4}N_2}(\tau) = 2\, q^{1/6}\eta^3(2\tau)\nu(q^{1/2}) + \sqrt{2}\,\holoproj{\eta^{3}(2\tau)G_{4}(\tau)}.
\end{aligned}
\end{equation}
We next use the results from Section~\ref{sec:formulaholoproj} to determine the projection of each of the five functions
\begin{equation} \label{eqn:functions_to_be_projected}
\eta^3\Big({\tau/ 2}\Big)G_{2}(\tau), \qquad \eta^3(2\tau) G_{2}(\tau), \qquad \eta^3\Big({\tau/2}\Big)G_{3}(\tau), \qquad \eta^3\Big({\tau/2}\Big)G_{4}(\tau), \qquad \eta^3(2\tau)G_{4}(\tau)
\end{equation}
in the form of a double series. 

Recall from~\eqref{eqn:G_(3)_def}
\begin{equation*}
    G_{(3)}(\tau)= -\frac{i}{2\sqrt{6}}\int_{-\overline{\tau}}^{i\infty}\frac{g_{(3)}(z)}{\sqrt{-i(z+\tau)}}\,\dd z.
\end{equation*}
From the definition of the components of $g_{(3)}$ in~\eqref{eqn:g_(3)_components} and the definition of the theta functions $\theta_{N,a}$ in~\eqref{eqn:unary_theta_def}, it follows that
%From the definitions~\eqref{eqn:unary_theta_def} of the theta function $\theta_{N,a}$ and the components of $g_{(3)}$ in terms of theta functions, it is apparent that
\begin{align*} 
&
g_{(3),2}(\tau) = \sum_{n=1}^\infty \left({n \over 6}\right) n e^{{i\pi n^2 \over 24}\tau}, \\
&
g_{(3),3}(\tau) = \sum_{n=1}^\infty \left({n \over 12}\right) n e^{{i\pi n^2 \over 24} \tau }, \\
&
g_{(3),4}(\tau) = 4\sqrt{2}\sum_{n=1}^\infty (-1)^n \left({n \over 3}\right) n e^{{2\pi i n^2 \over 3}\tau }. 
\end{align*}
Substituting these expansions into the definition of $G_{(3)}$ gives
\begin{align*}
G_{2}(\tau) &= -{i \over 2\sqrt{6}} \sum_{n=1}^\infty \left({n \over 6}\right) n \int_{-\bar{\tau}}^{i\infty} {e^{{i\pi n^2 \over 24}z} \over \sqrt{-i(z+\tau)}} \, \dd z, \\
G_{3}(\tau) &= -{i \over 2\sqrt{6}} \sum_{n=1}^\infty \left({n \over 12}\right) n \int_{-\bar{\tau}}^{i\infty} {e^{{i\pi n^2 \over 24}z} \over \sqrt{-i(z+\tau)}} \, \dd z, \\
G_{4}(\tau) &= -{2i \over \sqrt{3}} \sum_{n=1}^\infty (-1)^n \left({n \over 3}\right) n \int_{-\bar{\tau}}^{i\infty} {e^{{2i\pi n^2 \over 3}z} \over \sqrt{-i(z+\tau)}} \, \dd z.
\end{align*}
We may, moreover, write $\eta$ using Jacobi's identity~\eqref{eqn:Jacobi_identity} in the form
\[ \eta^3(\tau) = \sum_{n=1}^\infty \left({-4 \over n}\right) n e^{{i\pi n^2 \over 4}\tau}. 
\]

The first function listed in~\eqref{eqn:functions_to_be_projected} is therefore
\begin{equation} \label{eqn:first_function_to_be_projected}
\eta^3\Big({\tau/ 2}\Big)G_{2}(\tau) = -{i \over 2\sqrt{6}} \bigg( \sum_{m=1}^\infty \left({-4 \over m}\right) m e^{{i\pi m^2 \over 8}\tau} \bigg) \sum_{n=1}^\infty \left({n \over 6}\right) n \int_{-\bar{\tau}}^{i\infty} {e^{{i\pi  n^2 \over 24}z} \over \sqrt{-i(z+\tau)}} \, \dd z.
\end{equation}
To this we may apply Theorem~\ref{thm:hol_proj_eval} with $a=1$, $b=16$, $c=1$ and $d=48$. The Dirichlet characters are $\chi_1(m) = \left({-4 \over m}\right)$ of modulus $p_1=4$, and $\chi_2(n) = \left({n\over6}\right)$ of modulus $p_2=24$. We take $\varepsilon_1=\varepsilon_2=1$. The value of $g$ is $\gcd(16/4,48/24)=2$, so $abcd/g^2=192$. The Pell equation to be solved is therefore $m^2-192n^2=1$; its fundamental solution is $(m,n)=(97,7)$, so $\alpha=97$ and $\beta=7$. Since $(\alpha+1)/p_2$ is not an integer, and $\chi_2$ is an odd character, we may use the formula~\eqref{eqn:hol_proj_odd} to find the holomorphic projection of~\eqref{eqn:first_function_to_be_projected}. The result is that
\[ 
\holoproj{ \eta^3\Big({\tau / 2}\Big)G_{2}(\tau) }= \sum_{m=1}^\infty \sum_{n=1}^{\lfloor {12m \over 7}\rfloor} \left({-4 \over m}\right) \left({n \over 6}\right) \Big( m-{7 \over 12}n \Big) q^{{m^2 \over 16}-{n^2 \over 48}}. 
\]
It happens that this is not the simplest form obtainable for the double series. A simplification is obtained by applying the first identity of Lemma~\ref{lem:double_ser_trans} with the function $\varphi$ given by $\varphi(m,n) = (m-\tfrac{7}{12}n)q^{m^2/16-n^2/48}$, so that the left-hand side of~\eqref{eqn:double_ser_trans1} agrees with the double series just obtained. Noting that the quadratic form in the exponent of $q$ is invariant under $(m,n) \mapsto (7m-4n,12m-7n)$, we find that the right-hand side of~\eqref{eqn:double_ser_trans1} yields the simpler expression
\[ 
\holoproj{\eta^3\Big({\tau / 2}\Big)G_{2}(\tau)} = \sum_{m=1}^\infty \sum_{n=1}^{\lfloor {3m \over 2}\rfloor} \left({-4 \over m}\right) \left({n \over 6}\right) \Big( m-{2 \over 3}n \Big) q^{{m^2 \over 16}-{n^2 \over 48}}. \]

For the second function in~\eqref{eqn:functions_to_be_projected}, which we may write in the form
\[ \eta^3(2\tau) G_{2}(\tau) = -{i \over 2\sqrt{6}} \bigg( \sum_{m=1}^\infty \left({-4 \over m}\right) m e^{{i\pi m^2 \over 2}\tau} \bigg) \sum_{n=1}^\infty \left({n \over 6}\right) n \int_{-\bar{\tau}}^{i\infty} {e^{{i\pi n^2 \over 24}z} \over \sqrt{-i(z+\tau)}} \, \dd z, \]
the process is almost identical. The characters $\chi_1$ and $\chi_2$ are unchanged, as are the values of $\varepsilon_1$, $\varepsilon_2$, $a$, $c$ and $d$. However, we now take $b=4$ so $g=\gcd(4/4,48/24)=1$. The value of $abcd/g^2$ is once again $192$, so the values of $\alpha$ and $\beta$ are the same as before. The formula~\eqref{eqn:hol_proj_odd} yields
\[ 
\holoproj{\eta^3(2\tau) G_{2}(\tau)} = \sum_{m=1}^\infty \sum_{n=1}^{\lfloor {24m \over 7} \rfloor} \left({-4 \over m}\right) \left({n \over 6}\right) \Big(m - {7 \over 24}n \Big) q^{{m^2 \over 4} - {n^2 \over 48}}. 
\]
This may be simplified using the second identity of Lemma~\ref{lem:double_ser_trans} since the quadratic form in the exponent of $q$ is again invariant under $(m,n) \mapsto (7m-2n,24m-7n)$. The result of this simplification is that
\[ 
\holoproj{\eta^3(2\tau) G_{2}(\tau)} = \sum_{m=1}^\infty \sum_{n=1}^{3m} \left({-4 \over m}\right) \left({n \over 6}\right) \Big(m - {1 \over 3}n \Big) q^{{m^2 \over 4} - {n^2 \over 48}}. 
\]

The third function in~\eqref{eqn:functions_to_be_projected} may be expressed as
\[ 
\eta^3\Big({\tau / 2}\Big)G_{3}(\tau) = -{i \over 2\sqrt{6}} \bigg( \sum_{m=1}^\infty \left({-4 \over m}\right) m e^{{i\pi m^2 \over 8}\tau} \bigg) \sum_{n=1}^\infty \left({n \over 12}\right) n \int_{-\bar{\tau}}^{i\infty} {e^{{i\pi n^2 \over 24}z} \over \sqrt{-i(z+\tau)}} \, \dd z. 
\]
We may determine the holomorphic projection of this function using Theorem~\ref{thm:hol_proj_eval} with $\varepsilon_1=\varepsilon_2=1$, $a=1$, $b=16$, $c=1$ and $d=48$. The Dirichlet characters are $\chi_1(m)=\left({-4\over m}\right)$ of modulus $p_1=4$, and $\chi_2(n) = \left({n\over 12}\right)$ of modulus $p_2=6$. The value of $g$ is $\gcd(16/4,48/6)=4$, so the Pell equation to be solved is $m^2-48n^2=1$. The fundamental solution is $(m,n)=(7,1)$, so $\alpha=7$ and $\beta=1$. Since $\chi_2$ is odd, we use the formula~\eqref{eqn:hol_proj_odd} from Theorem~\ref{thm:hol_proj_eval}. The theorem is applicable since $(\alpha+1)/p_2$ is not an integer. This gives
\[ 
\holoproj{\eta^3\Big({\tau / 2}\Big)G_{3}(\tau)} = \sum_{m=1}^\infty \sum_{n=1}^{\lfloor {3m \over 2}\rfloor} \left({-4 \over m}\right) \left({n \over 12}\right) \Big( m-{1 \over 2}n \Big) q^{{m^2 \over 16}-{n^2 \over 48}}. \]

The fourth and fifth functions in~\eqref{eqn:functions_to_be_projected} are
\begin{align}
\eta^3\Big({\tau \over 2}\Big)G_{4}(\tau)&= -{2i \over \sqrt{3}} \bigg( \sum_{m=1}^\infty \left({-4 \over m}\right) m e^{{i\pi m^2 \over 8}\tau} \bigg) \sum_{n=1}^\infty (-1)^n \left({n \over 3}\right) n \int_{-\bar{\tau}}^{i\infty} {e^{{2i\pi n^2 \over 3}z} \over \sqrt{-i(z+\tau)}} \, \dd z, \label{eqn:fourth_function_to_be_projected} \\
\eta^3(2\tau)G_{4}(\tau) &= -{2i \over \sqrt{3}} \bigg( \sum_{m=1}^\infty \left({-4 \over m}\right) m e^{{i\pi m^2 \over 2}\tau} \bigg) \sum_{n=1}^\infty (-1)^n \left({n \over 3}\right) n \int_{-\bar{\tau}}^{i\infty} {e^{{2i\pi n^2 \over 3}z} \over \sqrt{-i(z+\tau)}} \, \dd z. \label{eqn:fifth_function_to_be_projected}
\end{align}
In order to obtain the holomorphic projections of these, we apply Theorem~\ref{thm:hol_proj_eval} with $\varepsilon_1=1$, $\varepsilon_2=-1$, $a=1$, $c=1$, $d=3$ and the characters $\chi_1$ and $\chi_2$ given by $\chi_1(m)=\left({-4\over m}\right)$ and $\chi_2(n)=\left({n\over 3}\right)$. These are of moduli $p_1=4$ and $p_2=3$ respectively.
In the case of~\eqref{eqn:fourth_function_to_be_projected}, we take $b=16$. The value of $g$ is then $\gcd(16/4,3/3)=1$ and the Pell equation is again $m^2-48n^2=1$. As before, this yields $\alpha=7$ and $\beta=1$. Again, $(\alpha+1)/p_2$ is not an integer. The values of $\varepsilon_3$ and $\varepsilon_4$ given by~\eqref{eqn:epsilon_values} are $\varepsilon_3=-1$ and $\varepsilon_4=1$. Since the character $\chi_2$ is odd, the formula~\eqref{eqn:hol_proj_odd} is again used to obtain
\[ 
\holoproj{\eta^3\Big({\tau / 2}\Big)G_{4}(\tau)} = \sqrt{2} \sum_{m=1}^\infty \sum_{n=1}^{\lfloor {3m \over 8}\rfloor} (-1)^n \left({-4 \over m}\right) \left({n \over 3}\right) \Big( m-{7-(-1)^m \over 3}n \Big) q^{{m^2 \over 16}-{n^2 \over 3}}. \]
Since it is only the odd values of $m$ which make a non-zero contribution to the sum, it is possible to replace $(-1)^m$ with $-1$ inside the summand. Hence
\[ 
\holoproj{\eta^3\Big({\tau / 2}\Big)G_{4}(\tau)} = \sqrt{2} \sum_{m=1}^\infty \sum_{n=1}^{\lfloor {3m \over 8}\rfloor} (-1)^n \left({-4 \over m}\right) \left({n \over 3}\right) \Big( m-{8 \over 3}n \Big) q^{{m^2 \over 16}-{n^2 \over 3}}. 
\]
In the case of~\eqref{eqn:fifth_function_to_be_projected}, we take $b=4$. The value of $g$ is $\gcd(4/4,3/3)=1$, so $abcd/g^2=12$. The Pell equation is therefore $m^2-12n^2=1$; its fundamental solution is $(m,n)=(7,2)$, so $\alpha=7$ and $\beta=2$. Once again, $(\alpha+1)/p_2$ is not an integer. From~\eqref{eqn:epsilon_values}, the values of $\varepsilon_3$ and $\varepsilon_4$ are $1$. Then formula~\eqref{eqn:hol_proj_odd} yields
\[ \pi_\text{hol}\Big( \eta^3(2\tau) G_{4}(\tau) \Big) = \sqrt{2} \sum_{m=1}^\infty \sum_{n=1}^{\lfloor {3m \over 4}\rfloor} (-1)^n \left({-4 \over m}\right) \left({n \over 3}\right) (m-n)q^{{m^2 \over 4}-{n^2 \over 3}}. \]

Combining these with~\eqref{eqn:holproj_holproj} and Corollary~\ref{cor:holoprojids}, it follows that
\begin{equation*}
\begin{aligned}
&
\frac{1}{6}\frac{\eta^7(\tau)}{\eta^3(2\tau)} = q^{-1/48}\eta^3(\tau/2)\psi(q^{1/2}) + \frac{1}{2}\sum_{m=1}^\infty \sum_{n=1}^{\lfloor {3m \over 2}\rfloor} \left({-4 \over m}\right) \left({n \over 6}\right) \Big( m-{2 \over 3}n \Big) q^{{m^2 \over 16}-{n^2 \over 48}},\\
&
\frac{1}{3}\frac{\eta^7(\tau)}{\eta^3(\tau/2)} = q^{-1/48}\eta^3(2\tau)\psi(q^{1/2}) + \frac{1}{2}\,\sum_{m=1}^\infty \sum_{n=1}^{3m} \left({-4 \over m}\right) \left({n \over 6}\right) \Big(m - {1 \over 3}n \Big) q^{{m^2 \over 4} - {n^2 \over 48}},\\
&
\frac{1}{4}\frac{\eta^6(\tau/2)}{\eta^2(\tau)} = q^{-1/48}\eta^3(\tau/2)\psi(-q^{1/2}) + \frac{1}{2}\sum_{m=1}^\infty \sum_{n=1}^{\lfloor {3m \over 2}\rfloor} \left({-4 \over m}\right) \left({n \over 12}\right) \Big( m-{1 \over 2}n \Big) q^{{m^2 \over 16}-{n^2 \over 48}},\\
&
\frac{4}{3}\frac{\eta^3(\tau/2)\eta^3(2\tau)}{\eta^2(\tau)} = q^{1/6}\eta^3(\tau/2)\nu(q^{1/2}) + \sum_{m=1}^\infty \sum_{n=1}^{\lfloor {3m \over 8}\rfloor} (-1)^n\! \left({-4 \over m}\right) \left({n \over 3}\right) \Big( m-{8 \over 3}n \Big) q^{{m^2 \over 16}-{n^2 \over 3}},\\
&
\frac{\eta^6(2\tau)}{\eta^2(\tau)} = q^{1/6}\eta^3(2\tau)\nu(q^{1/2}) + \sum_{m=1}^\infty \sum_{n=1}^{\lfloor {3m \over 4}\rfloor} (-1)^n \left({-4 \over m}\right) \left({n \over 3}\right) (m-n)q^{{m^2 \over 4}-{n^2 \over 3}}.
\end{aligned}
\end{equation*}

Plugging in the definition of $\eta$ and replacing $q$ by $q^2$ yields the respective identities
\begin{equation*}
\begin{aligned}
&
\frac{1}{6}q^{1/12}\frac{(q^2;q^2)^7_\infty}{(q^4;q^4)^3_\infty} = q^{1/12} (q;q)^3_\infty\psi(q) + \frac{1}{2}\sum_{m=1}^\infty \sum_{n=1}^{\lfloor {3m \over 2}\rfloor} \left({-4 \over m}\right) \left({n \over 6}\right) \Big( m-{2 \over 3}n \Big) q^{{m^2 \over 8}-{n^2 \over 24}},\\
&
\frac{1}{3}q^{11/24}\frac{(q^2;q^2)^7_\infty}{(q;q)^3_\infty} = q^{11/24}(q^4;q^4)_\infty^3\psi(q) + \frac{1}{2}\,\sum_{m=1}^\infty \sum_{n=1}^{3m} \left({-4 \over m}\right) \left({n \over 6}\right) \Big(m - {1 \over 3}n \Big) q^{{m^2 \over 2} - {n^2 \over 24}},\\
&
\frac{1}{4}q^{1/12}\frac{(q;q)^6_\infty}{(q^2;q^2)^2_\infty} = q^{1/12}(q;q)^3_\infty\psi(-q) + \frac{1}{2}\sum_{m=1}^\infty \sum_{n=1}^{\lfloor {3m \over 2}\rfloor} \left({-4 \over m}\right) \left({n \over 12}\right) \Big( m-{1 \over 2}n \Big) q^{{m^2 \over 8}-{n^2 \over 24}},\\
&
\frac{4}{3}q^{11/24}\frac{(q;q)^3_\infty(q^4;q^4)^3_\infty}{(q^2;q^2)^2_\infty} = q^{11/24}(q;q)^3_\infty\nu(q) + \sum_{m=1}^\infty \sum_{n=1}^{\lfloor {3m \over 8}\rfloor} (-1)^n\! \left({-4 \over m}\right) \left({n \over 3}\right) \Big( m-{8 \over 3}n \Big) q^{{m^2 \over 8}-{2n^2 \over 3}},\\
&
q^{5/6}\frac{(q^4;q^4)^6_\infty}{(q^2;q^2)^2_\infty} = q^{5/6}(q^4;q^4)^3_\infty\nu(q) + \sum_{m=1}^\infty \sum_{n=1}^{\lfloor {3m \over 4}\rfloor} (-1)^n \left({-4 \over m}\right) \left({n \over 3}\right) (m-n)q^{{m^2 \over 2}-{2n^2 \over 3}}.
\end{aligned}
\end{equation*}
The proof is completed upon dividing both sides of these identities by
$q^{1/12}$, $q^{11/24}$, $q^{1/12}$, $q^{11/24}$ and $q^{5/6}$ respectively.
\end{proof}

For convenience of the reader, Table~\ref{table:Thm 4.2} summarizes the parameters used in applying Theorem~\ref{thm:hol_proj_eval} to derive identities~\eqref{eqn:psi2}--\eqref{eqn:nu3}.

\setlength{\tabcolsep}{7pt}   % default is 6pt
\begin{table}[H]
\caption{}
\label{table:Thm 4.2}
\centering    
\begin{tabular}{|c|c|c|c|c|c|c|c|c|c|c|c|c|c|}
\hline
\textit{Identity} & $a$ & $b$ & $c$ & $d$ & $\chi_1$ & $\chi_2$ & $p_1$ & $p_2$ & $\epsilon_1$ & $\epsilon_2$ & \textit{Pell equation} & $\alpha$ & $\beta$ \\[1pt]
\hline
\ref{eqn:psi2} & $1$ & $16$ & $1$ & $48$ & $\left(\frac{-4}{\vphantom{6}\cdot}\right)$ & $\left(\frac{\vphantom{-4}\cdot}{6}\right)$ & $4$ & $24$ & $1$ & $1$ & $m^2 - 192n^2 =1$ & $97$ & $7$ \\[1pt]
\hline
\ref{eqn:psi3} & $1$ & $4$ & $1$ & $48$ & $\left(\frac{-4}{\vphantom{6}\cdot}\right)$ & $\left(\frac{\vphantom{-4}\cdot}{6}\right)$ & $4$ & $24$ & $1$ & $1$ & $m^2 - 192n^2 =1$ & $97$ & $7$ \\
\hline
\ref{eqn:psi5} & $1$ & $16$ & $1$ & $48$ & $\left(\frac{-4}{\vphantom{6}\cdot}\right)$ & $\left(\frac{\vphantom{-4}\cdot}{12}\right)$ & $4$ & $6$ & $1$ & $1$ & $m^2 - 48n^2 =1$ & $7$ & $1$ \\
\hline
\ref{eqn:nu2} & $1$ & $16$ & $1$ & $3$ & $\left(\frac{-4}{\vphantom{6}\cdot}\right)$ & $\left(\frac{\vphantom{-4}\cdot}{3}\right)$ & $4$ & $3$ & $1$ & $-1$ & $m^2 - 48n^2 =1$ & $7$ & $1$ \\
\hline
\ref{eqn:nu3} & $1$ & $4$ & $1$ & $3$ & $\left(\frac{-4}{\vphantom{6}\cdot}\right)$ & $\left(\frac{\vphantom{-4}\cdot}{3}\right)$ & $4$ & $3$ & $1$ & $-1$ & $m^2 - 12n^2 =1$ & $7$ & $2$ \\
\hline
\end{tabular}
\end{table}
% ----------------------------------------------------------------------------------------------------------------------------------------------------------------------------------------------------------------------------------------------------------------------
\newpage
%Appendix 
\appendix
%Finding the subrepresentations using Mathematica
\renewcommand{\thesection}{\Roman{section}}
\section{\textbf{Procedural Guide and Numerical Data}}
\label{ProcGuide}
\renewcommand{\thesubsection}{\thesection.\alph{subsection}}
\subsection{Finding $\sigma_1,...,\sigma_6$:}\label{ap: rep}
We describe how we found the representation spaces in Lemma \ref{lem: Rep}.
We study the action of $S, T$ on the components of $H_{(3)} \otimes N$.

Recall
\[
H_{(3)} = (H_0, H_1, H_2, H_3, H_4, H_5)^T,
\qquad
N = (N_0, N_1, N_2)^T,
\]
and
\[
H_{(3)} \otimes N
= \bigl(H_0 N_0,\, H_0 N_1,\, H_0 N_2,\, H_1 N_0,\, \ldots,\, H_5 N_2\bigr)^T.
\]

We find that
\begin{align*}
H_0 N_0 (\tau+1) &= \zeta_{24}\, H_1 N_1(\tau), \\
H_1 N_1 (\tau+1) &= \zeta_{24}\, H_0 N_0(\tau), \\
H_0 N_0\!\left(-\tfrac{1}{\tau}\right) &= -\tau^2\, H_2 N_0(\tau), \\
H_2 N_0\!\left(-\tfrac{1}{\tau}\right) &= -\tau^2\, H_0 N_0(\tau).
\end{align*}

We represent this in the diagram
\[
\begin{tikzcd}[column sep=large]
H_2 N_0 \arrow[r, leftrightarrow, "S"] & H_0 N_0 \arrow[r, leftrightarrow, "T"] & H_1 N_1.
\end{tikzcd}
\]

Similarly, we have
\[
\begin{tikzcd}[column sep=large, row sep=large]
H_0 N_0 \arrow[r, leftrightarrow, "S"] \arrow[d, leftrightarrow, "T"']
  & H_2 N_0 \arrow[d, leftrightarrow, "T"] \\
H_1 N_1 \arrow[d, leftrightarrow, "S"']
  & H_3 N_1 \arrow[d, leftrightarrow, "S"] \\
H_5 N_2 \arrow[r, leftrightarrow, "T"']
  & H_4 N_2.
\end{tikzcd}
\]

This implies
\[
\operatorname{Span}\bigl( H_0 N_0,\, H_2 N_0,\, H_1 N_1,\, H_3 N_1,\, H_5 N_2,\, H_4 N_2 \bigr)
\]
is a six dimensional space invariant under $SL_2(\mathbb{Z})$.

\bigskip

Considering the action on $H_0 N_0 + H_2 N_0$, we find
\[
\begin{tikzcd}[column sep=huge]
H_0 N_0 + H_2 N_0
  \arrow[r, leftrightarrow, "T"]
  \arrow[loop left, "S"]
& H_1 N_1 + H_3 N_1
  \arrow[r, leftrightarrow, "S"]
& H_4 N_2 + H_5 N_2
  \arrow[loop right, "T."]
\end{tikzcd}
\]

In fact, defining
\[
W_1 = \bigl[\, H_0 N_0 + H_2 N_0,\;\; H_1 N_1 + H_3 N_1,\;\; H_4 N_2 + H_5 N_2 \,\bigr]^T,
\]
we have
\[
W_1(\tau)\big|_{2}\, \gamma = \sigma_1(\gamma)\, W_1(\tau) \qquad \text{for } \gamma \in Mp_2(\mathbb{Z}),
\]
where
\[
\sigma_1(T) =
\begin{pmatrix}
0 & \zeta_{24} & 0 \\
\zeta_{24} & 0 & 0 \\
0 & 0 & \zeta_{12}^{5}
\end{pmatrix},
\qquad
\sigma_1(S) =
\begin{pmatrix}
-1 & 0 & 0 \\
0 & 0 & -1 \\
0 & -1 & 0
\end{pmatrix}.
\]

Thus
\[
\operatorname{Span}\bigl( H_0 N_0 + H_2 N_0,\; H_1 N_1 + H_3 N_1,\; H_4 N_2 + H_5 N_2 \bigr)
\]
is a three dimensional space invariant under $SL_2(\mathbb{Z})$. 
The three functions
\[
H_0 N_0 + H_2 N_0,\quad H_1 N_1 + H_3 N_1,\quad H_4 N_2 + H_5 N_2
\]
correspond to the 3 columns of the first $3 \times 18$ matrix in Lemma~\ref{lem: Rep} giving the representation space of $\sigma_1$. 
The other matrices in Lemma~\ref{lem: Rep} are found in a similar fashion. 
So we define
\begin{align*}
W_1&=\left[H_{0} N_{0}+H_{2} N_{0}, 
H_{1} N_{1}+H_{3} N_{1}, 
H_{4} N_{2}+H_{5} N_{2}\right]^T, \\
W_2&=\left[H_{0} N_{0}-H_{2} N_{0}, 
H_{1} N_{1}-H_{3} N_{1}, 
H_{4} N_{2}-H_{5} N_{2}\right]^T, \\
W_3&=\left[H_{0} N_{2}+H_{4} N_{0}, 
H_{1} N_{2}+H_{5} N_{1}, 
H_{2} N_{1}+H_{3} N_{0}\right]^T,\\
W_4&=\left[H_{0} N_{2}-H_{4} N_{0}, 
H_{1} N_{2}-H_{5} N_{1}, 
H_{2} N_{1}-H_{3} N_{0}\right]^T,\\
W_5&=\left[H_{0} N_{1}+H_{1} N_{0}, 
H_{2} N_{2}+H_{5} N_{0}, 
H_{3} N_{2}+H_{4} N_{1}\right]^T,\\
W_6&=\left[H_{0} N_{1}-H_{1} N_{0}, 
H_{2} N_{2}-H_{5} N_{0}, 
H_{3} N_{2}-H_{4} N_{1}\right]^T.
\end{align*}
Then we have
\[
W_j(\tau)\big|_{2}\, \gamma = \sigma_j(\gamma)\, W_j(\tau) \qquad \text{for } \gamma \in Mp_2(\mathbb{Z}),
\]
where the matrices $\sigma_j(T)$ and $\sigma_j(S)$ are given in Table
\ref{table}.  
This gives the required irreducible subrepresentations $\sigma_j$.

%Table -----------------------------------------------------------------------------------------

\newcommand{\spacerrow}{\rule{0pt}{2ex} & & & & & \\}
\begin{table}[htbp]
\caption{}
\label{table}
\centering

{\scriptsize
\setlength{\tabcolsep}{4pt}
\renewcommand{\arraystretch}{1}

\begin{tabular}{|c|c|c|c|c|c|}
\hline
$j$ & $\sigma_j(T)$ & $\sigma_j(S)$ & eigen$_{j,T}$ & eigen$_{j,S}$ & eigen$_{j,ST}$ \\
\hline

\spacerrow
1 &
$\left(\begin{array}{ccc}
0 & \zeta_{24} & 0 \\
\zeta_{24} & 0 & 0 \\
0 & 0 & \zeta_{12}^{5}
\end{array}\right)$
&
$\left(\begin{array}{ccc}
-1 & 0 & 0 \\
0 & 0 & -1 \\
0 & -1 & 0
\end{array}\right)$
&
$e^{5\pi i/6},\,e^{\pi i/12},\,e^{13\pi i/12}$
&
$-1,-1,1$
&
$1,e^{2\pi i/3},e^{4\pi i/3}$
\\
\spacerrow
\hline

\spacerrow
2 &
$\left(\begin{array}{ccc}
0 & \zeta_{24} & 0 \\
\zeta_{24} & 0 & 0 \\
0 & 0 & \zeta_{12}^{11}
\end{array}\right)$
&
$\left(\begin{array}{ccc}
1 & 0 & 0 \\
0 & 0 & -1 \\
0 & -1 & 0
\end{array}\right)$
&
$e^{11\pi i/6},\,e^{\pi i/12},\,e^{13\pi i/12}$
&
$-1,1,1$
&
$1,e^{2\pi i/3},e^{4\pi i/3}$
\\
\spacerrow
\hline

\spacerrow
3 &
$\left(\begin{array}{ccc}
0 & \zeta_{48}^{11} & 0 \\
\zeta_{48}^{11} & 0 & 0 \\
0 & 0 & \zeta_{24}
\end{array}\right)$
&
$\left(\begin{array}{ccc}
0 & 0 & -1 \\
0 & -1 & 0 \\
-1 & 0 & 0
\end{array}\right)$
&
$e^{\pi i/12},\,e^{11\pi i/24},\,e^{35\pi i/24}$
&
$-1,-1,1$
&
$1,e^{2\pi i/3},e^{4\pi i/3}$
\\
\spacerrow
\hline

\spacerrow
4 &
$\left(\begin{array}{ccc}
0 & \zeta_{48}^{11} & 0 \\
\zeta_{48}^{11} & 0 & 0 \\
0 & 0 & \zeta_{24}^{13}
\end{array}\right)$
&
$\left(\begin{array}{ccc}
0 & 0 & -1 \\
0 & 1 & 0 \\
-1 & 0 & 0
\end{array}\right)$
&
$e^{13\pi i/12},\,e^{11\pi i/24},\,e^{35\pi i/24}$
&
$-1,1,1$
&
$1,e^{2\pi i/3},e^{4\pi i/3}$
\\
\spacerrow
\hline

\spacerrow
5 &
$\left(\begin{array}{ccc}
\zeta_{24} & 0 & 0 \\
0 & 0 & \zeta_{48}^{11} \\
0 & \zeta_{48}^{11} & 0
\end{array}\right)$
&
$\left(\begin{array}{ccc}
0 & -1 & 0 \\
-1 & 0 & 0 \\
0 & 0 & -1
\end{array}\right)$
&
$e^{\pi i/12},\,e^{11\pi i/24},\,e^{35\pi i/24}$
&
$-1,-1,1$
&
$1,e^{2\pi i/3},e^{4\pi i/3}$
\\
\spacerrow
\hline

\spacerrow
6 &
$\left(\begin{array}{ccc}
\zeta_{24}^{13} & 0 & 0 \\
0 & 0 & \zeta_{48}^{11} \\
nj0 & \zeta_{48}^{11} & 0
\end{array}\right)$
&
$\left(\begin{array}{ccc}
0 & -1 & 0 \\
-1 & 0 & 0 \\
0 & 0 & 1
\end{array}\right)$
&
$e^{13\pi i/12},\,e^{11\pi i/24},\,e^{35\pi i/24}$
&
$-1,1,1$
&
$1,e^{2\pi i/3},e^{4\pi i/3}$
\\
\spacerrow
\hline

\end{tabular}
}
\end{table}
% ------------------------------------------------------------------------------------------------------------------------------------------------------------
\newpage
\subsection{SageMath verification:}\label{A2}
We provide the SageMath code computing the Taylor expansions of the eighth components of $\holoproj{H_{(3)}\otimes N}$ and $E$; the remaining components may be treated analogously.
Together with the theoretical arguments in the proof of Theorem~\ref{thm: equality}, this gives the leading coefficients on both sides, proving the identity
$$
\holoproj{H_{(3)}\otimes N}=E.
$$

\textit{SageMath computations}: 
\begin{figure}[ht]
\flushleft
\includegraphics[ width=1.1\textwidth,
height=0.75\textheight, trim=2cm 7cm 0cm 2.5cm,
clip]{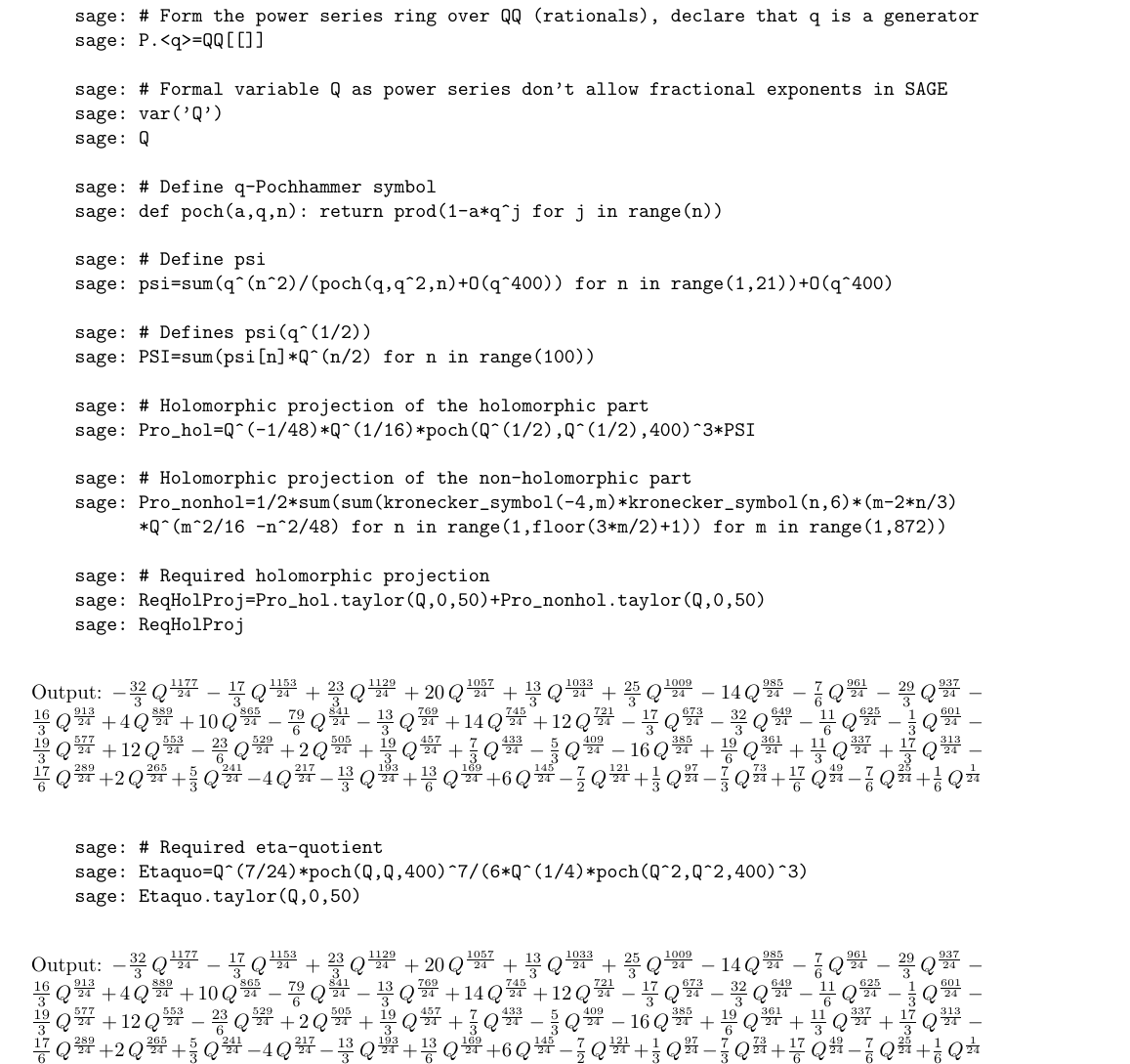}
\end{figure}

% Appendix II
\newpage
\section{\textbf{The Eighteen Convolution Identities}}
\label{app: identities}
\subsection{A catalog of identities from holomorphic projection} \label{app:all id}
In this section, we record the complete list of the eighteen convolution-type $q$-series identities arising from the eighteen holomorphic projection identities obtained in Theorem~\ref{thm: equality}.
In Section~\ref{sec:proofmain}, we derived five of these identities in
detail which proved our main result. The remaining thirteen identities are obtained similarly by applying Theorem~\ref{thm:hol_proj_eval}.
%------------------------------------------------------------

\bigskip
\textbf{Identities for $\phi$:}
\begin{enumerate}[label=II.\arabic*.,ref=II.\arabic*, itemsep=1em]
\item $\displaystyle
\begin{aligned}
\frac{(q^2;q^2)_\infty^9}{(q;q)_\infty^3\,(q^4;q^4)_\infty^3}\,\phi(q)
&=
\frac{1}{2}\,\frac{(q^2;q^2)_\infty^{16}}{(q;q)_\infty^6\,(q^4;q^4)_\infty^6}
+\sum_{m=1}^\infty\sum_{n=1}^{\lfloor \frac{3}{2}m \rfloor}
\left(\frac{-2}{m}\right)\left(\frac{n}{6}\right)\left(m-\frac{1}{2}n\right)
q^{\frac{m^2}{8}-\frac{n^2}{24}-\frac{1}{12}}
\end{aligned}\label{1}
$
\item $\displaystyle
\begin{aligned}
(q;q)_\infty^3\,\phi(q)
&=
\frac{2}{3}\,\frac{(q^2;q^2)_\infty^7}{(q^4;q^4)_\infty^3}
+\sum_{m=1}^\infty\sum_{n=1}^{\lfloor \frac{3}{2}m \rfloor}
\left(\frac{-4}{m}\right)\left(\frac{n}{6}\right)\left(m-\frac{2}{3}n\right)
q^{\frac{m^2}{8}-\frac{n^2}{24}-\frac{1}{12}}
\end{aligned}\label{2}
$
\item $\displaystyle
\begin{aligned}
(q^4;q^4)_\infty^3\,\phi(q)
&=
\frac{1}{3}\,\frac{(q^2;q^2)_\infty^7}{(q;q)_\infty^3}
+\sum_{m=1}^\infty\sum_{n=1}^{3m}
\left(\frac{-4}{m}\right)\left(\frac{n}{6}\right)\left(m-\frac{1}{3}n\right)
q^{\frac{m^2}{2}-\frac{n^2}{24}-\frac{11}{24}}
\end{aligned}\label{3}
$
\item $\displaystyle
\begin{aligned}
\frac{(q^2;q^2)_\infty^9}{(q;q)_\infty^3\,(q^4;q^4)_\infty^3}\,\phi(-q)
&=
\frac{2}{3}\,\frac{(q^2;q^2)_\infty^{7}}{(q^4;q^4)_\infty^3}
+\sum_{m=1}^\infty\sum_{n=1}^{\lfloor \frac{3}{2}m \rfloor}
\left(\frac{-2}{m}\right)\left(\frac{n}{12}\right)\left(m-\frac{2}{3}n\right)
q^{\frac{m^2}{8}-\frac{n^2}{24}-\frac{1}{12}}
\end{aligned}\label{4}
$
\item $\displaystyle
\begin{aligned}
(q;q)_\infty^3\,\phi(-q)
&=
\frac{1}{2}\,\frac{(q;q)_\infty^6}{(q^2;q^2)_\infty^2}
+\sum_{m=1}^\infty\sum_{n=1}^{\lfloor \frac{3}{2}m \rfloor}
\left(\frac{-4}{m}\right)\left(\frac{n}{12}\right)\left(m-\frac{1}{2}n\right)
q^{\frac{m^2}{8}-\frac{n^2}{24}-\frac{1}{12}}
\end{aligned}\label{5}
$
\item $\displaystyle
\begin{aligned}
(q^4;q^4)_\infty^3\,\phi(-q)
&=
\frac{1}{3}\,\frac{(q;q)_\infty^3\,(q^4;q^4)_\infty^3}{(q^2;q^2)_\infty^2}
+\sum_{m=1}^\infty\sum_{n=1}^{3m}
\left(\frac{-4}{m}\right)\left(\frac{n}{12}\right)\left(m-\frac{1}{3}n\right)
q^{\frac{m^2}{2}-\frac{n^2}{24}-\frac{11}{24}}
\end{aligned}\label{6}
$
\end{enumerate}
%------------------------------------------------------------

\bigskip
\textbf{Identities for $\psi$:}
\begin{enumerate}[label=II.\arabic*.,ref=II.\arabic*, itemsep=1em]
\setcounter{enumi}{6}
\item $\displaystyle
\begin{aligned}
\frac{(q^2;q^2)_\infty^9}{(q;q)_\infty^3\,(q^4;q^4)_\infty^3}\,\psi(q)
&=
\frac{1}{4}\,\frac{(q^2;q^2)_\infty^{16}}{(q;q)_\infty^6\,(q^4;q^4)_\infty^6}
-\frac{1}{2}\sum_{m=1}^\infty \sum_{n=1}^{\lfloor {3m \over 2} \rfloor} \left( {-2 \over m} \right) \left( {n \over 6} \right) \Big( m - \frac{1}{2}n \Big) q^{{m^2 \over 8} - {n^2 \over 24} - {1 \over 12}}
\end{aligned}\label{7}
$
\item $\displaystyle
\begin{aligned}[t]
(q;q)_\infty^3\,\psi(q)
&=
\frac{1}{6}\,\frac{(q^2;q^2)_\infty^7}{(q^4;q^4)_\infty^3} - \frac{1}{2}\,
\sum_{m=1}^\infty\sum_{n=1}^{\left\lfloor \frac{3}{2}m \right\rfloor} \left(\frac{-4}{m}\right)\left(\frac{n}{6}\right)\left(m-\frac{2}{3}n\right)
q^{\frac{m^2}{8}-\frac{n^2}{24}-\frac{1}{12}}
\end{aligned}\label{8}
$
\item $\displaystyle
\begin{aligned}
(q^4;q^4)_\infty^3\,\psi(q)
&=
\frac{1}{3}\,\frac{(q^2;q^2)_\infty^7}{(q;q)_\infty^3}
-\frac{1}{2}\sum_{m=1}^\infty\sum_{n=1}^{3m}
\left(\frac{-4}{m}\right)\left(\frac{n}{6}\right)\left(m-\frac{1}{3}n\right)
q^{\frac{m^2}{2}-\frac{n^2}{24}-\frac{11}{24}}
\end{aligned}\label{9}
$
\item $\displaystyle
\begin{aligned}
\frac{(q^2;q^2)_\infty^9}{(q;q)_\infty^3\,(q^4;q^4)_\infty^3}\,\psi(-q)
&=
\frac{1}{6}\,\frac{(q^2;q^2)_\infty^{7}}{(q^4;q^4)_\infty^3}
-\frac{1}{2}\sum_{m=1}^\infty\sum_{n=1}^{\lfloor \frac{3}{2}m \rfloor}
\left(\frac{-2}{m}\right)\left(\frac{n}{12}\right)\left(m-\frac{2}{3}n\right)
q^{\frac{m^2}{8}-\frac{n^2}{24}-\frac{1}{12}}
\end{aligned}\label{10}
$
\item $\displaystyle
\begin{aligned}
(q;q)_\infty^3\,\psi(-q)
&=
\frac{1}{4}\,\frac{(q;q)_\infty^6}{(q^2;q^2)_\infty^2}-\frac{1}{2}\,
\sum_{m=1}^\infty\sum_{n=1}^{\left\lfloor \frac{3}{2}m \right\rfloor}
\left(\frac{-4}{m}\right)\left(\frac{n}{12}\right)\left(m-\frac{1}{2}n\right)
q^{\frac{m^2}{8}-\frac{n^2}{24}-\frac{1}{12}}
\end{aligned}\label{11}
$
\item $\displaystyle
\begin{aligned}
(q^4;q^4)_\infty^3\,\psi(-q)
&=
\frac{1}{3}\,\frac{(q;q)_\infty^3\,(q^4;q^4)_\infty^3}{(q^2;q^2)_\infty^2}
-\frac{1}{2}\sum_{m=1}^\infty\sum_{n=1}^{3m}
\left(\frac{-4}{m}\right)\left(\frac{n}{12}\right)\left(m-\frac{1}{3}n\right)
q^{\frac{m^2}{2}-\frac{n^2}{24}-\frac{11}{24}}
\end{aligned}\label{12}
$
\end{enumerate}
%------------------------------------------------------------

\bigskip
\textbf{Identities for $\nu$:}
\begin{enumerate}[label=II.\arabic*.,ref=II.\arabic*, itemsep=1em]
\setcounter{enumi}{12}
\item \(\displaystyle
\begin{aligned}[t]
\frac{(q^2;q^2)_\infty^9}{(q;q)_\infty^3\,(q^4;q^4)_\infty^3}\,\nu(q)
&= \frac{2}{3}\,\frac{(q^2;q^2)_\infty^{7}}{(q;q)_\infty^3}
-\sum_{m=1}^\infty\sum_{n=1}^{\lfloor \frac{3}{8}m \rfloor}
(-1)^n \left(\frac{-2}{m}\right)\left(\frac{n}{3}\right)\left(m-\frac{8}{3}n\right)
q^{\frac{m^2}{8}-\frac{2n^2}{3}-\frac{11}{24}}
\end{aligned}\label{13}
\)
\item \(\displaystyle
\begin{aligned}[t]
(q;q)_\infty^3\,\nu(q)
&=
\frac{4}{3}\,
\frac{(q;q)_\infty^3\,(q^4;q^4)_\infty^3}
{(q^2;q^2)_\infty^2}
-\sum_{m=1}^\infty
\sum_{n=1}^{\left\lfloor \frac{3}{8}m \right\rfloor}
(-1)^n
\left(\frac{-4}{m}\right)
\left(\frac{n}{3}\right)
\left(m-\frac{8}{3}n\right)
q^{\frac{m^2}{8}-\frac{2n^2}{3}-\frac{11}{24}}
\end{aligned}\label{14}
\)
\item \(\displaystyle
\begin{aligned}[t]
(q^4;q^4)_\infty^3\,\nu(q)
&=
\frac{(q^4;q^4)_\infty^6}
{(q^2;q^2)_\infty^2}
-\sum_{m=1}^\infty
\sum_{n=1}^{\left\lfloor \frac{3m}{4} \right\rfloor}
(-1)^{n}
\left(\frac{-4}{m}\right)
\left(\frac{n}{3}\right)
(m-n)
q^{\frac{m^2}{2}-\frac{2n^2}{3}-\frac{5}{6}}
\end{aligned}\label{15}
\)
\item \(\displaystyle
\begin{aligned}[t]
\frac{(q^2;q^2)_\infty^9}{(q;q)_\infty^3\,(q^4;q^4)_\infty^3}\,\nu(-q)
&= \frac{4}{3}\,\frac{(q^2;q^2)_\infty^{7}}{(q;q)_\infty^3}
+\sum_{m=1}^\infty\sum_{n=1}^{\lfloor \frac{3}{8}m \rfloor}
(-1)^n \left(\frac{-2}{m}\right)\left(\frac{n}{3}\right)\left(m-\frac{8}{3}n\right)
q^{\frac{m^2}{8}-\frac{2n^2}{3}-\frac{11}{24}}
\end{aligned}\label{16}
\)
\item \(\displaystyle
\begin{aligned}[t]
(q;q)_\infty^3\,\nu(-q)
&= \frac{2}{3}\,\frac{(q;q)_\infty^3\,(q^4;q^4)_\infty^3}{(q^2;q^2)_\infty^2}
+\sum_{m=1}^\infty
\sum_{n=1}^{\left\lfloor \frac{3}{8}m \right\rfloor}
(-1)^n
\left(\frac{-4}{m}\right)
\left(\frac{n}{3}\right)
\left(m-\frac{8}{3}n\right)
q^{\frac{m^2}{8}-\frac{2n^2}{3}-\frac{11}{24}}
\end{aligned}\label{17}
\)
\item \(\displaystyle
\begin{aligned}[t]
(q^4;q^4)_\infty^3\,\nu(-q)
&= \frac{(q^4;q^4)_\infty^6}{(q^2;q^2)_\infty^2}
+\sum_{m=1}^\infty
\sum_{n=1}^{\left\lfloor \frac{3m}{4} \right\rfloor}
(-1)^{n}
\left(\frac{-4}{m}\right)
\left(\frac{n}{3}\right)
(m-n)
q^{\frac{m^2}{2}-\frac{2n^2}{3}-\frac{5}{6}}
\end{aligned}\label{18}
\)
\end{enumerate}

\begin{remark} \label{rmk:proj_lin_comb}
A striking feature of our full list is that certain pairs of identities share both the same infinite product and the same double series, but with different coefficients. 
This is the case, for instance, for the identities~\ref{1} and~\ref{7}. If we add these two identities together, so as to eliminate the series, we obtain the classical relation~\eqref{eqn:phi-psi_relation}. 
If instead we eliminate this product between them, we obtain the identity~\eqref{eqn:Bc} of Corollary~\ref{cor:alt_main} but with $q \mapsto -q$. Similarly, adding~\ref{13} and~\ref{16} gives relation~\ref{eqn:even_part_of_nu}.
\end{remark}

%\bigskip
\subsection{Equivalences among the identities}
\label{sec:otherids}

The purpose of this section is to show that each of the identities in Appendix~\ref{app:all id} is equivalent to one of the five identities stated in Theorem~\ref{thm: main}. 
Note that the identities \eqref{eqn:psi2} -- \eqref{eqn:nu3} in Theorem~\ref{thm: main} are the identities \ref{8}, \ref{9}, \ref{11}, \ref{14}, and \ref{15} respectively in Appendix~\ref{app:all id}.

The first six identities \ref{1} -- \ref{6} involving the function $\phi$ are readily seen to be obtained from the identities~\ref{7} -- \ref{12} using relation~\eqref{eqn:phi-psi_relation}. 
Therefore, we begin by showing that each identity~\ref{7}, \ref{10}, and \ref{12} are equivalent to one of the identities in Theorem~\ref{thm: main}.

Observe that the identities~\eqref{eqn:psi5} and~\ref{7} are equivalent by changing $q \mapsto -q$. For instance, changing $q \mapsto -q$ in~\eqref{eqn:psi5} puts it in the form
\[ (-q;-q)_\infty^3\,\psi(q)=\frac{1}{4}\,\frac{(-q;-q)_\infty^6}{(q^2;q^2)_\infty^2}+\frac{1}{2}\,\sum_{m=1}^\infty\sum_{n=1}^{\lfloor \frac{3}{2}m \rfloor}\left(\frac{-4}{m}\right)\left(\frac{n}{12}\right)\left(m-\frac{1}{2}n\right)(-q)^{\frac{m^2}{8}-\frac{n^2}{24}-\frac{1}{12}}. \]
Now, the left-hand side of this equation is the same as that in~\ref{7} since
\begin{equation} \label{eqn:q-Pochhammer_sign_change}
(-q;-q)_\infty = (-q;q^2)_\infty (q^2;q^2)_\infty = {(q^2;q^2)_\infty^3 \over (q;q)_\infty (q^4;q^4)_\infty}.
\end{equation}
On the right-hand side, we may simplify the double series with the observation that, for all integers $m$ and $n$,
\begin{equation} \label{eqn:Kronecker_symbol_identities}
(-1)^{m^2-1 \over 8}\left( {-4 \over m} \right) = \left( {-2 \over m} \right) \qquad \text{and} \qquad (-1)^{n^2-1 \over 24} \left( {n \over 12} \right) = \left( {n \over 6} \right).
\end{equation}
Hence, the result of changing $q \mapsto -q$ in~\eqref{eqn:psi5} is
\[ \frac{(q^2;q^2)_\infty^9}{(q;q)_\infty^3\,(q^4;q^4)_\infty^3}\,\psi(q)=\frac{1}{4}\,\frac{(q^2;q^2)_\infty^{16}}{(q;q)_\infty^6\,(q^4;q^4)_\infty^6}-{1 \over 2}\sum_{m=1}^\infty \sum_{n=1}^{\lfloor {3m \over 2} \rfloor} \left( {-2 \over m} \right) \left( {n \over 6} \right) \Big( m - {n \over 2} \Big) q^{{m^2 \over 8} - {n^2 \over 24} - {1 \over 12}}, \]
which is the identity~\ref{7}. 
In much the same way, the identities~\eqref{eqn:psi2} and~\ref{10} are seen to be equivalent under the substitution $q \mapsto -q$. 
Likewise, \eqref{eqn:psi3} and~\ref{12} are equivalent. 
In the latter case, one must note that
\[ (-1)^{{m^2 -1} \over 2} \left( {-4 \over m} \right) \]
may be replaced by $\left({-4\over m}\right)$, since the Kronecker symbol $\left({-4\over m}\right)$ vanishes unless m is odd. For odd $m$, we have $m^2-1 \equiv 0$ (mod $4$), and hence
\[
(-1)^{{m^2 -1} \over 2} =1.
\]

We are now left with only the identities for $\nu$. 
% Recall from~\eqref{eqn:even_part_of_nu} that the functions $\nu(q)$ and $\nu(-q)$ are related by
% \[ \nu(q) + \nu(-q) = 2{(q^4;q^4)^3_\infty \over (q^2;q^2)^2_\infty}. \]
From the relation~\eqref{eqn:even_part_of_nu}, it follows that~\eqref{eqn:nu2} is equivalent to~\ref{17}, \eqref{eqn:nu3} is equivalent to~\ref{18}, and~\ref{13} is equivalent to~\ref{16}. The identities~\ref{13} and~\ref{17} are equivalent to each other by changing $q \mapsto -q$, as are~\eqref{eqn:nu2} and~\ref{16}. 
To see this, we must once again make use of~\eqref{eqn:q-Pochhammer_sign_change} and the first of the relations in~\eqref{eqn:Kronecker_symbol_identities}. 
It is also necessary to observe that for every integer $n$,
\[ (-1)^{2n^2+1 \over 3} \left({n \over 3}\right) = -\left({n \over 3}\right). \]

%Appendix III
\newpage
\section{\textbf{Elementary proofs}}
\label{sec:eleproofs}

There is an alternative approach that could, in principle, have been used in order to prove the identities in Theorem~\ref{thm: main} without using holomorphic projection. Consider, for example, that the double series
\begin{equation} \label{eqn:two-var_double_sum_for_id_a}
\sum_{m=1}^\infty \sum_{n=1}^{\left\lfloor \frac{3}{2}m \right\rfloor} \left(\frac{-4}{m}\right) \left(\frac{n}{6}\right) \left(x^{m-\frac{2}{3}n} - x^{-m+\frac{2}{3}n} \right) q^{\frac{m^2}{8}-\frac{n^2}{24}-\frac{1}{12}}
\end{equation}
may be written in terms of the function
\[ f_{a,b,c}(x,y;q)\coloneqq \sum_{m,n=0}^\infty (-1)^{m+n} x^m y^n q^{{am(m-1) \over 2} + bmn + {cn(n-1) \over 2}} - \sum_{m,n=1}^\infty (-1)^{m+n} x^{-m} y^{-n} q^{{am(m+1) \over 2} + bmn + {cn(n+1) \over 2}} \]
which is the subject of Hickerson and Mortenson's paper~\cite{Hickerson2014}. By splitting the sums over $m$ and $n$ into residue classes modulo $8$ and $12$ respectively, it is straightforward to determine that the sum~\eqref{eqn:two-var_double_sum_for_id_a} is equal to
\begin{align*}
&x^{1/3} f_{4,4,1}\left(-q^{10} x^8,q^3;q^4\right)
+q x^{1/3} f_{4,4,1}\left(-q^{18} x^8,q^5;q^4\right)
+q^2 x^{5/3} f_{4,4,1}\left(-q^{18} x^8,q^7;q^4\right) \\
&+q^5 x^{5/3} f_{4,4,1}\left(-q^{26} x^8,q^9;q^4\right)
-q x^{7/3} f_{4,4,1}\left(-q^{14} x^8,q^7;q^4\right)
-q^4 x^{7/3} f_{4,4,1}\left(-q^{22} x^8,q^9;q^4\right) \\
&-q^5 x^{11/3} f_{4,4,1}\left(-q^{22} x^8,q^{11};q^4\right)
-q^{10} x^{11/3} f_{4,4,1}\left(-q^{30} x^8,q^{13};q^4\right)
+q^3 x^{13/3} f_{4,4,1}\left(-q^{18}x^8,q^{11};q^4\right) \\
&+q^8 x^{13/3} f_{4,4,1}\left(-q^{26} x^8,q^{13};q^4\right)
+q^9 x^{17/3} f_{4,4,1}\left(-q^{26} x^8,q^{15};q^4\right)
+q^{16} x^{17/3} f_{4,4,1}\left(-q^{34} x^8,q^{17};q^4\right) \\
&-q^6 x^{19/3} f_{4,4,1}\left(-q^{22} x^8,q^{15};q^4\right)
-q^{13} x^{19/3} f_{4,4,1}\left(-q^{30} x^8,q^{17};q^4\right)
-q^{14} x^{23/3} f_{4,4,1}\left(-q^{30} x^8,q^{19};q^4\right) \\
&-q^{23}x^{23/3} f_{4,4,1}\left(-q^{38}x^8,q^{21};q^4\right).
\end{align*}
The function $f_{4,4,1}(x,y;q)$ may be evaluated in closed form in terms of theta functions and Appell--Lerch series using~\cite[Theorem~1.4]{Hickerson2014} or~\cite[Theorem~6.1]{Mortenson2023}. From the latter, it follows that
{\fontsize{9.5}{11}\selectfont
\begin{align*}
&f_{4,4,1}(x,y;q) = \frac{\theta(y;q)}{\theta\left(-z;q^{12}\right)} \sum _{n=-\infty }^{\infty } \frac{q^{6 n (n-1)} z^n}{1+x y^{-4} z q^{6(2n-1)}} \\
&+\frac{1}{\theta \left(-z;q^{12}\right)}\sum _{m=0}^{15} \sum _{k=0}^3 (-1)^{m+k} q^{2 m (m-1)+4 k m+\frac{1}{2} k (k-1)} x^m y^k \theta \left(x z q^{4 (m+k)};q^{16}\right) \sum _{n=-\infty }^{\infty } \frac{q^{384 n^2+48 (m+k) n} x^{12 n} z^{-4 n}}{1+x^{-1}y^4z^{-1} q^{6 (32n+2m+1)}},
\end{align*}
}
where the value of the variable $z$ on the right-hand side is arbitrary. After differentiating both sides of the resulting identity and setting $x=1$, we may obtain a formula for the double series in~\eqref{eqn:psi2}. This converts~\eqref{eqn:psi2} to an identity involving theta functions, Appell--Lerch series, and their derivatives, which we may then attempt to verify. However, the number of terms involved seems to render this approach impractical.

Another disadvantage of this alternative approach is that it is not well suited to discovering identities of this type. It relies on advance knowledge of the form of the double series, which our approach in Sections~\ref{sec:transrep}--\ref{sec:proofmain} did not require. Double series of the form
\[ \sum_{m=1}^\infty \sum_{n=1}^{\lfloor \nu m \rfloor} (\lambda m + \mu n) q^{Q(m,n)}, 
\]
where $\lambda$, $\mu$ and $\nu$ are real numbers and $Q$ is a quadratic form, are typically expressible not solely in terms of theta functions and Appell--Lerch series, since the derivatives of these functions are required to express them. The fact that the double series in Theorem~\ref{thm: main} are so expressible is due to cancellation between the derivative terms.

\subsection{The identity~\eqref{eqn:nu3}}
There is another alternative approach that we may take which makes use of double series identities for the rank function. In the case of the identity~\eqref{eqn:nu3} it gives us an elementary proof, which does not require the form of the identity to be known in advance. This derivation of~\eqref{eqn:nu3} goes as follows.

Let
\[ \theta(x;q) \coloneqq \sum_{n=-\infty}^\infty (-1)^n q^{n(n-1) \over 2} x^n = (x,q/x,q;q)_\infty \qquad \text{and} \qquad R(x;q) \coloneqq\sum_{n=0}^\infty {q^{n^2} \over (qx,q/x;q)_n}. \]
The identity~\eqref{eqn:nu3} can be obtained by purely elementary manipulations from the following identity from~\cite[Equation~2.20]{Garvan2023}:
\[ {\theta(x^2;q) \over 1-x} R(x;q) = \sum_{m=0}^\infty \sum_{n=-\lfloor{m \over 2}\rfloor}^{\lfloor{m \over 2}\rfloor} (-1)^{m+n} (x^{m+1}+x^{-m}) q^{{m(m+1) \over 2} - {n(3n+1) \over 2}}. \]
Differentiate both sides of this equation with respect to $x$ and set $x=\sqrt{q}$ to obtain
\begin{equation} \label{eqn:derivative_at_q}
{2(q)_\infty^3 \over 1-\sqrt{q}} R(\sqrt{q};q) = \sum_{m=0}^\infty \sum_{n=-\lfloor{m \over 2}\rfloor}^{\lfloor{m \over 2}\rfloor} (-1)^{m+n} \Big( (m+1)q^{m+1 \over 2}-mq^{-{m \over 2}}\Big) q^{{m(m+1) \over 2} - {n(3n+1) \over 2}}.
\end{equation}

The identity
\[ \sum_{n=1}^\infty {q^{n(n-1)} \over (x,q/x;q)_n} = {1 \over x}\bigg( {R(x;q) \over 1-x} - 1 \bigg) \]
is given by Gordon and McIntosh in their survey~\cite[Equation~6.4]{Gordon2012}. From this it follows that the mock theta function $\omega$ defined by~\eqref{eqn:omega_def} is expressible as
\[ \omega(q) = {1 \over q} \bigg( {R(q;q^2) \over 1-q} - 1 \bigg). \]
Hence, after changing $q \mapsto q^2$, the equation~\eqref{eqn:derivative_at_q} may be written in terms of $\omega$ as
\begin{equation} \label{eqn:omega_expansion}
(q^2;q^2)_\infty^3 \big( q\omega(q) + 1\big) = {1 \over 2}\sum_{m=0}^\infty \sum_{n=-\lfloor{m \over 2}\rfloor}^{\lfloor{m \over 2}\rfloor} (-1)^{m+n} \Big( (m+1)q^{m+1}-mq^{-m}\Big) q^{m(m+1) - n(3n+1)}.
\end{equation}

The mock theta functions $\nu$ and $\omega$ are related by the equation~\eqref{eqn:nu-omega_relation}.
%\[ \nu(q) + q\omega(q^2) = (-q^2;q^2)_\infty^3 (q^2;q^2)_\infty. \]
The equation~\eqref{eqn:omega_expansion} may therefore be written in terms of $\nu$ as
{\fontsize{10}{11}\selectfont
\begin{equation} \label{eqn:initial_nu_expansion}
(q^4;q^4)_\infty^3 \nu(q) - {(q^4;q^4)_\infty^6 \over (q^2;q^2)_\infty^2} - {(q^4;q^4)_\infty^3 \over q} = -{1 \over 2}\sum_{m=0}^\infty \sum_{n=-\lfloor{m \over 2}\rfloor}^{\lfloor{m \over 2}\rfloor} (-1)^{m+n} \Big( (m+1)q^{2(m+1)}-mq^{-2m}\Big) q^{2m(m+1) - 2n(3n+1)-1}.
\end{equation}
}
It is straightforward to verify that
\[ \sum_{n=-\lfloor{m \over 2}\rfloor}^{\lfloor{m \over 2}\rfloor} (-1)^n q^{-2n(3n+1)} = \sum_{n=1}^{3\lfloor {m \over 2} \rfloor + 2} (-1)^{n-1} \left({n \over 3} \right) q^{-{2 \over 3}(n-1)(n-2)} - 1 \]
simply by splitting the sum on the right-hand side of this identity into its residue classes modulo $6$. With this, we may rewrite~\eqref{eqn:initial_nu_expansion} in the form
\begin{align}
(q^4;q^4)_\infty^3 \nu(q) - {(q^4;q^4)_\infty^6 \over (q^2;q^2)_\infty^2} - {(q^4;q^4)_\infty^3 \over q} &= -{1 \over 2}\sum_{m=0}^\infty (-1)^m \Big( (m+1)q^{2(m+1)}-mq^{-2m}\Big) q^{2m(m+1)-1} \label{eqn:modified_nu_expansion} \\
&\quad \times \Bigg( \sum_{n=1}^{3\lfloor {m \over 2} \rfloor + 2} (-1)^{n-1} \left({n \over 3} \right) q^{-{2 \over 3}(n-1)(n-2)} - 1 \Bigg). \nonumber
\end{align}
If we now split the right-hand side into two series by multiplying out its factor of $(m+1)q^{2(m+1)}-mq^{-2m}$, and then change $m \mapsto m-1$ in the first of these series and recombine them, it becomes
\[ {1 \over 2}\sum_{m=1}^\infty (-1)^m m q^{2m^2-1}\Bigg( \sum_{n=1}^{3\lfloor {m \over 2} \rfloor + 2} (-1)^{n-1} \left({n \over 3} \right) q^{-{2 \over 3}(n-1)(n-2)} + \sum_{n=1}^{3\lfloor {m-1 \over 2} \rfloor + 2} (-1)^{n-1} \left({n \over 3} \right) q^{-{2 \over 3}(n-1)(n-2)} - 2 \Bigg). \]
If $m$ is odd, then the two inner summations are identical. If, on the other hand, $m$ is even, then the first of them contains three terms which are absent from the second. One of these, corresponding to $n=3m/2$, is zero. The other two are $(-1)^{m \over 2} q^{-m({3m \over 2}-1)}$ and $(-1)^{m \over 2} q^{-m({3m \over 2} + 1)}$. Hence the right-hand side of~\eqref{eqn:modified_nu_expansion} equals
\begin{align*}
&\sum_{m=1}^\infty \sum_{n=1}^{3\lfloor {m \over 2} \rfloor + 2} (-1)^{m+n-1} \left({n \over 3} \right) m q^{2m^2-{2 \over 3}(n-1)(n-2)-1} - \sum_{m=1}^\infty (-1)^m m q^{2m^2-1} \\
&- {1 \over 2}\sum_{m=1}^\infty 2m q^{8m^2-1} \Big( (-1)^{m} q^{-2m(3m-1)} + (-1)^{m} q^{-2m(3m+1)} \Big).
\end{align*}
Since, by~\eqref{eqn:Jacobi_identity}, we may expand the product $(q^4;q^4)_\infty^3$ in series as
\[ (q^4;q^4)_\infty^3 = \sum_{m=0}^\infty (-1)^m (2m+1) q^{2m(m+1)}, \]
it follows that the identity~\eqref{eqn:modified_nu_expansion} may be written as
\begin{equation} \label{eqn:nu_expansion_with_S}
(q^4;q^4)_\infty^3 \nu(q) - {(q^4;q^4)_\infty^6 \over (q^2;q^2)_\infty^2} = S + \sum_{m=0}^\infty (-1)^m q^{2m^2-1} \Big( (2m+1)q^{2m} - m(q^{2m} + 1 + q^{-2m}) \Big),
\end{equation}
where $S$ denotes the double series
\begin{equation} \label{eqn:S_def}
S \coloneqq \sum_{m=1}^\infty \sum_{n=1}^{\lfloor {3m \over 2} \rfloor + 2} (-1)^{m+n-1} \left({n \over 3} \right) m q^{2m^2-{2 \over 3}(n-1)(n-2)-1}.
\end{equation}
The replacement of the previous upper limit of $3\lfloor {m \over 2} \rfloor + 2$ with $\lfloor {3m \over 2} \rfloor + 2$ may be justified by considering that, although this introduces extra terms into the double sum, all of these extra terms correspond to values of $n$ which are multiples of $3$ and therefore make no contribution to the value of the sum. We may simplify the single series on the right-hand side of~\eqref{eqn:nu_expansion_with_S} by observing that
\[ \sum_{m=0}^\infty (-1)^m q^{2m^2-1} \Big( (m+1) q^{2m} - m q^{-2m} \Big) = 2\sum_{m=1}^\infty (-1)^{m-1} m q^{2m(m-1)-1}. \]
Hence,
\begin{equation} \label{eqn:simplified_nu_expansion_with_S}
(q^4;q^4)_\infty^3 \nu(q) - {(q^4;q^4)_\infty^6 \over (q^2;q^2)_\infty^2} = S + \sum_{m=1}^\infty (-1)^{m-1} m q^{2m^2-1} (1+2q^{-2m}).
\end{equation}

In the double series~\eqref{eqn:S_def}, we choose to make the substitution
\[ (m,n) \mapsto (2m-n+1,3m-2n+3). \]
This is invertible; it is, in fact, self-inverse, so the summation indices remain integer-valued. Moreover, the inequalities
\[ m \geq 1, \qquad n \geq 1, \qquad n \leq {3m \over 2} + 2, \]
which may be taken to define the region over which the summation is carried out in~\eqref{eqn:S_def}, are transformed into
\[ n \leq 2m, \qquad n \leq {3m \over 2} + 1, \qquad n \geq -1 \]
respectively. The first and third of these imply that $m \geq 0$. Moreover, when $m \geq 1$, we may dispense with the condition that $n \leq 2m$ since this is implied by the second inequality.
Hence, we may write the double series $S$ as
\[ S = \sum_{m=0}^\infty \sum_{n=-1}^{\lfloor{3m \over 2}\rfloor + 1} (-1)^{m+n-1} \left( {n \over 3} \right) (2m - n + 1) q^{{(2m+1)^2 \over 2} - {2n^2 \over 3} - {5 \over 6}}. \]
In general, it is obvious that
\[ \sum_{m=0}^\infty (-1)^m \varphi(m) = \sum_{m=1}^\infty \left( {-4 \over m} \right) \varphi\bigg( {m-1 \over 2} \bigg), \]
whenever $\varphi$ is such that these series converge. Hence, $S$ may be written as
\begin{equation} \label{eqn:simplified_S}
S = \sum_{m=1}^\infty \sum_{n=-1}^{\lfloor{3m+1 \over 4}\rfloor} (-1)^{n-1} \left( {-4 \over m} \right) \left( {n \over 3} \right) (m-n) q^{{m^2 \over 2} - {2n^2 \over 3} - {5 \over 6}}.
\end{equation}

For an expression of the form
\begin{equation} \label{eqn:generic_double_series}
\sum_{m=1}^\infty \sum_{n=-1}^{\lfloor{3m+1 \over 4}\rfloor} \varphi(m,n),
\end{equation}
in which $\varphi$ simply denotes any function for which the double series converges absolutely, the upper limit of the inner summation is always equal to $\lfloor 3m/4 \rfloor$, except when $m \equiv 1 \text{ (mod $4$)}$, in which case it equals $\lfloor 3m/4 \rfloor + 1$. We may therefore write~\eqref{eqn:generic_double_series} in the equivalent form
\[ \sum_{m=1}^\infty \sum_{n=0}^{\lfloor{3m \over 4}\rfloor} \varphi(m,n) + \sum_{m=1}^\infty \varphi(m,-1) + \sum_{m=1}^\infty \varphi(4m-3,3m-2). \]
If we apply this transformation to the double series obtained for $S$ in~\eqref{eqn:simplified_S}, then we find that
{\fontsize{9.5}{11}\selectfont
\[ S = \sum_{m=1}^\infty \sum_{n=1}^{\lfloor{3m \over 4}\rfloor} (-1)^{n-1} \left( {-4 \over m} \right) \left( {n \over 3} \right) (m-n) q^{{m^2 \over 2} - {2n^2 \over 3} - {5 \over 6}} - \sum_{m=1}^\infty \left( {-4 \over m} \right) (m+1) q^{{m^2 \over 2} - {3 \over 2}} + \sum_{m=1}^\infty (-1)^{m-1} (m-1) q^{2m^2-4m+1}. \]}
Upon substituting this into~\eqref{eqn:simplified_nu_expansion_with_S}, we obtain~\eqref{eqn:nu3} after some straightforward manipulations.

\subsection{The identities~\eqref{eqn:psi2}, \eqref{eqn:psi3} and~\eqref{eqn:nu2}}

There is another identity for the rank function which can be used to obtain~\eqref{eqn:psi2} and a weaker form of the identities~\eqref{eqn:psi3} and~\eqref{eqn:nu2}. This identity, given in~\cite[Theorem~2.1]{Bradley-Thrush2024}, is
\begin{equation} \label{eqn:two-var_rank_HR_with_q^4}
{\theta(x;q) \over 1-x} R(x;q^4) = \sum_{n=1}^{\infty} \sum_{m=1}^{n} (-1)^n q^{2n^2-{m(3m-1) \over 2}} (1-q^{2m-1})(x^n+x^{-n})
+ \phi(-q) \theta(xq^2;q^4).
\end{equation}
In~\cite[p.~1956]{Bradley-Thrush2024}, it is observed that this identity admits the following two special cases~\cite[Equations~1.5 and~1.6]{Bradley-Thrush2024}, obtained by differentiating~\eqref{eqn:two-var_rank_HR_with_q^4} and setting $x=q^{-2}$ and $x=q$:
\begin{align}
&(q;q)_\infty^3 \left( \! \nu(-q) + {1 \over q} \right) \! - 2(q^4;q^4)_\infty^3 \psi(-q) \nonumber \\
&= \sum_{n=1}^{\infty} \sum_{m=1}^n (-1)^n n q^{2n(n-1)-{m(3m-1) \over 2}} (1-q^{2m-1})(1 - q^{4n}), \label{eqn:nu_and_psi} \\
&(q;q)_\infty^3 \big( \psi(q)+1 \big) + {1 \over 4} \big( \theta(q;q^4)-(-q;-q)_\infty^3 \big) \phi(-q) = \nonumber \\
&\sum_{n=1}^\infty \sum_{m=1}^n (-1)^{n-1} n q^{n(2n-1) - {m(3m-1) \over 2}} (1-q^{2m-1})(1-q^{2n}). \label{eqn:psi_and_phi}
\end{align}

By using the identity~\eqref{eqn:even_part_of_nu} for the even part of the mock theta function $\nu$, together with the elementary relation~\eqref{eqn:q-Pochhammer_sign_change}, it is straightforward to see that~\eqref{eqn:nu_and_psi} can be rewritten in the form
\begin{align*}
&\left( (q;q)_\infty^3 \nu(q) - {4 (q;q)_\infty^3 (q^4;q^4)_\infty^3 \over 3(q^2;q^2)_\infty^2} \right) + 2\left( (q^4;q^4)_\infty^3 \psi(-q) - {(q^2;q^2)_\infty^7 \over 3(-q;-q)_\infty^3} \right) \\
&= {1 \over q}(q;q)_\infty^3 - \sum_{n=1}^{\infty} \sum_{m=1}^n (-1)^n n q^{2n(n-1)-{m(3m-1) \over 2}} (1-q^{2m-1})(1 - q^{4n}).
\end{align*}
From this, it is apparent that~\eqref{eqn:nu_and_psi} is equivalent to a linear combination of the two identities~\eqref{eqn:psi3} and~\eqref{eqn:nu2}, with the substitution $q \mapsto -q$ having been made in the former of these.

As to the identity~\eqref{eqn:psi_and_phi}, the special case $x=1/q$ of~\eqref{eqn:two-var_rank_HR_with_q^4} tells us that
\[ \theta(q;q^4) \phi(-q) = \sum_{n=1}^\infty \sum_{m=1}^n (-1)^{n-1} q^{n(2n-1)-{m(3m-1) \over 2}} (1-q^{2m-1})(1+q^{2n}). \]
By using this expansion, together with the relation~\eqref{eqn:phi-psi_relation} between the mock theta functions $\phi$ and $\psi$, we may rewrite the identity~\eqref{eqn:psi_and_phi} in the form
\begin{align*}
&(q;q)_\infty^3 \psi(q) + {1 \over 2} (-q;-q)_\infty^3 \psi(-q) \\
&= {(q^2;q^2)_\infty^7 \over 4(q^4;q^4)_\infty^3} - (q;q)_\infty^3 + {1 \over 4} \sum_{n=1}^\infty \sum_{m=1}^n (-1)^{n-1}q^{n(2n-1) - {m(3m-1) \over 2}} (1-q^{2m-1})\big(4n-1-(4n+1)q^{2n}\big).
\end{align*}
If, in this identity, we change $q \mapsto -q$, we obtain an identity with a different linear combination of $(q;q)_\infty^3 \psi(q)$ and $(-q;-q)_\infty^3 \psi(-q)$ on its left-hand side. We may then eliminate the second of these functions between the two identities to deduce that
\begin{align*}
(q;q)_\infty^3 \psi(q) &= {(q^2;q^2)_\infty^7 \over 6(q^4;q^4)_\infty^3} - {4 \over 3}(q;q)_\infty^3 + {2 \over 3}(-q;-q)_\infty^3 \\
&\quad + {1 \over 6} \sum_{n=1}^\infty \sum_{m=1}^n (-1)^{n-1}q^{n(2n-1) - {m(3m-1) \over 2}} \big(2(1-q^{2m-1}) - (-1)^{n-{m(3m-1) \over 2}}(1+q^{2m-1}) \big) \\
&\qquad \times \big(4n-1-(4n+1)q^{2n}\big).
\end{align*}
This is an alternative form of the identity~\eqref{eqn:psi2}, with its right-hand side written in a significantly more cumbersome form. It can presumably be reduced to the simpler form given in~\eqref{eqn:psi2} by means of manipulations similar to those used in the previous subsection.

\begin{remark}
There are other identities, not contained in Theorem~\ref{thm: main}, but of a similar type, that we may obtain from expansions of the rank function. For example, the expansion
\[ {\theta(x;q) \over 1-x} R\big(x;q^2\big) = \sum_{n=0}^\infty \sum_{m=-n}^n (-1)^m q^{{n(3n+1) \over 2}-m^2} x^m \big(1-q^{2n+1}\big), \]
given in~\cite[Equation~2.19]{Garvan2023}, yields the following identity for the mock theta function $\omega$ after differentiating both sides with respect to $x$ and then setting $x=q$:
\[ (q)_\infty^3\big(q \omega(q) + 1\big) = \sum_{n=0}^\infty \sum_{m=-n}^n (-1)^m m q^{{n(3n+1) \over 2} - m(m-1)} (1-q^{2n+1}). \]
It seems likely that this identity, with its double series written in a different form, should also be obtainable by using holomorphic projection and the result of Theorem~\ref{thm:hol_proj_eval}.
\end{remark}

\begin{remark}
The above shows that the identities~\eqref{eqn:psi2} and~\eqref{eqn:nu3} are essentially implicitly contained in known double series expansions of the rank function, although some quite substantial manipulations are required to convert from one form to the other. It would be interesting to know whether~\eqref{eqn:psi3}, \eqref{eqn:psi5} and~\eqref{eqn:nu2} can be obtained in a similar manner. Since we obtain only one particular linear combination of~\eqref{eqn:psi3} and~\eqref{eqn:nu2} from~\eqref{eqn:nu_and_psi}, it seems likely that further identities for the rank function would be needed in order to do so.
\end{remark}

\bibliographystyle{abbrv}
\bibliography{ref.bib}

\end{document}